\documentclass[a4paper,12pt,leqno,english]{article}
\usepackage[utf8]{inputenc}
\usepackage[T1]{fontenc}
\usepackage{babel}
\usepackage{amsmath}
\usepackage{amssymb}
\usepackage{enumitem}
\usepackage{graphicx}
\usepackage{hyperref}
\usepackage{tcolorbox}
\usepackage{pictex}
\usepackage{comment}
\usepackage{tikz} %

 \newcommand{\B}{{\mathbb  B}}
 \newcommand{\C}{{\mathbb  C}}

 \newcommand{\D}{{\mathbb D}}
 \newcommand{\F}{{\mathcal F}}
 \newcommand{\G}{{\mathcal G}}
 \renewcommand{\H}{{\mathcal H}}
 \renewcommand{\L}{{\mathcal L}}
 \newcommand{\N}{{\mathbb  N}}
 \newcommand{\OO}{{\mathcal O}}
 
 \newcommand{\Q}{{\mathbb  Q}}
 \newcommand{\R}{{\mathbb  R}}
 
 \newcommand{\T}{{\mathbb  T}}
 \newcommand{\Z}{{\mathbb  Z}}

 \newcommand{\scalar}[2]{{\langle#1,#2\rangle}}

 \newcommand{\Exp}{{\operatorname{Exp}}}
 \renewcommand{\Im}{{\operatorname{Im}}}
 \newcommand{\degree}{{\operatorname{deg}}}
 \newcommand{\ch}{{\operatorname{ch}}}
 \newcommand{\ext}{{\operatorname{ext}}}
\newcommand{\Loneloc}{{L_{\text{loc}}^1}}
 \newcommand{\Log}{{\operatorname{Log}}}
 \newcommand{\PSH}{{\operatorname{{\mathcal{PSH}}}}}
 \newcommand{\SH}{{\operatorname{{\mathcal {SH}}}}}
 \newcommand{\supp}{{\operatorname{supp}\, }}
 \newcommand{\USC}{{\operatorname{{\mathcal{USC}}}}}
 \newcommand{\LSC}{{\operatorname{{\mathcal{LSC}}}}}

\numberwithin{equation}{section}

\newtheorem{theorem+}           {Theorem}      [section]
\newtheorem{definition+}  [theorem+]  {Definition}
\newtheorem{lemma+}  [theorem+]  {Lemma}
\newtheorem{corollary+}  [theorem+]  {Corollary}
\newtheorem{proposition+}  [theorem+]  {Proposition}
\newtheorem{example+}  [theorem+]  {Example}
\newtheorem{problem+}  [theorem+]  {Problem}
\newtheorem{question+}  [theorem+]  {Question}
\newtheorem{remark+}  [theorem+]  {Remark}

\newenvironment{theorem}{\begin{theorem+}\sl}{\end{theorem+}\rm}
\newenvironment{definition}{\begin{definition+}\rm}{\end{definition+}\rm}
\newenvironment{lemma}{\begin{lemma+}\sl}{\end{lemma+}\rm}
\newenvironment{corollary}{\begin{corollary+}\sl}{\end{corollary+}\rm}
\newenvironment{proposition}{\begin{proposition+}\sl}{\end{proposition+}\rm}
\newenvironment{example}{\begin{example+}\rm}{\end{example+}\rm}

\newenvironment{proof}{\medbreak\noindent{\bf  Proof:}\rm}{\hfill$\square$\rm}
\newenvironment{prooftx}[1]{\medbreak\noindent{\bf 
    #1:}\rm}{\hfill$\square$\rm} 

\title{{\Large \bf 
Polynomials with exponents in  compact convex sets  and 
associated weighted extremal functions - \\
Fundamental results }}
\author{{Benedikt Steinar Magnússon, Álfheiður Edda Sigurðardóttir,} \\
{Ragnar Sigurðsson and Bergur Snorrason}}
\date{{}} 
\begin{document}
\maketitle

\begin{abstract} \noindent
This paper is a collection of fundamental results for the study of
polynomial rings $\mathcal P^S(\C^n)$ where the $m$-th degree
polynomials have exponents restricted to $mS$, where $S\subseteq
\R^n_+$ is compact, convex and $0\in S$. 
We study the relationship between $\mathcal P^S(\C^n)$ and the class
$\L^S(\C^n)$  of global plurisubharmonic functions where the growth is
determined by the logarithmic supporting function of $S$. 
We present properties of their respective weighted extremal functions
$\Phi_{K, q}^S$ and $V_{K, q}^S$ in connection with properties of
$S$. Our ambition is to give  detailed proofs  with minimal
assumptions of all results,  thus creating a self-contained
exposition.  
\end{abstract}

\section{Introduction}
\label{sec:01}

Approximation theory deals with problems of determining whether a
given function in some prescribed function space can be approximated 
by functions in a certain subspace.  The theorems of Runge and
Mergelyan are prototypes of results from approximation theory.
The Runge theorem states that every function $f$  holomorphic in some
neighborhood of a simply connected compact subset $K$ of $\C$
can be approximated in the uniform norm $\|\cdot\|_K$ 
on $K$ by polynomials and the Mergelyan
theorem states that it is enough to assume that $f$ is  continuous on 
$K$ and holomorphic in the interior of $K$.  

Quantitative approximation theory deals with problems of relating 
prop\-er\-ties of the given function $f$ to the error in approximation
which is usually measured as the distance from $f$ to a certain
finite dimensional subspaces of the approximating subspace. 
The Bernstein-Walsh theorem is a prototype
of a result from quantitative approximation theory.
It states that a holo\-morphic function $f$ defined in some neighborhood of a
compact simply connected  $K\subset\C$ extends as a holomorphic
function to the sublevel set 
$\{z\in \C\,;\, g_{\C\setminus K}(z,\infty)<\log R\}$ 
if and only if $\varlimsup_{m\to +\infty} d_m(f,K)^{1/m}\leq 1/R$,
where $R>1$, $g_{\C\setminus K}(\cdot,\infty)$ is the Green function  
of $\C\setminus K$ with logarithmic pole at $\infty$,
$d_m(f,K)$ is the distance from $f$ to the space ${\mathcal P}_m(\C)$
of all polynomials of degree $\leq m$ with respect to the uniform norm on $K$,
and it is assumed that the domain $\C\setminus K$ is regular for 
the Dirichlet problem in the sense that 
$g_{\C\setminus K}(\cdot,\infty)$ vanishes at the boundary of $K$.

The Runge theorem is generalized to several variables where simple 
connectedness of $K$ generalizes as polynomial convexity.
This generalization  is usually called the Oka-Weil theorem. 
There only exist fragmentary, but interesting, 
results generalizing the Mergelyan theorem to several
complex variables.
See Levenberg \cite{Lev:2006}.

The origin of the subject of the present paper is the generalization by
Siciak \cite{Sic:1962} of the Bernstein-Walsh theorem. 
There he introduced the extremal function 
$\Phi_K=\varlimsup_{m\to \infty}\Phi_{K,m}$, where
$\Phi_{K,m}=\sup\{|p|^{1/m} \,;\, p\in {\mathcal P}_m(\C^n),
\|p\|_K\leq 1\}$
and ${\mathcal P}_m(\C^n)$ is the space of all polynomials of degree
$\leq m$.    In his work
$\log\Phi_K$ plays the role of $g_{\C\setminus K}(\cdot,\infty)$
and the regularity condition at $\partial K$ 
is that $\Phi^*_K(z)=\varlimsup_{\zeta\to z} \Phi_K(\zeta)\leq 1$
for every $z\in \partial K$.  
Siciak's paper is a seminal work on the understanding of 
the interrelation between quantitative approximation theory in several
complex variables and  pluripotential theory. 
He studied this subject further in many of his works, for example
\cite{Sic:1981,Sic:1982,Sic:1990,Sic:1997}.
See also Bloom's Appendix B in the monograph by Saff and Totik 
\cite{SaffTotik:1997}.

We will systematically work with subspaces of the polynomial space
${\mathcal P}(\C^n)$ in $n$ variables.
For every non-empty bounded subset $S$ of $\R^n_+$  we 
let ${\mathcal P}^S_m(\C^n)$ denote 
the space of  all polynomials $p$ which can be written on the form 
$p(z)=\sum_{\alpha\in (mS)\cap \N^n} a_\alpha z^\alpha$, $z\in \C^n$,
and let  ${\mathcal P}^S(\C^n)=\bigcup_{m\in \N}{\mathcal P}^S_m(\C^n)$.
The standard simplex
$\Sigma=\ch\{0,e_1,\dots,e_n\}$, where $\ch\, A$ denotes 
the convex hull of a set $A$ and $(e_1,\dots,e_n)$ is the standard
basis in $\R^n$, yields the standard grading of
polynomials of degree $\leq m$, that is
${\mathcal P}^\Sigma_m(\C^n)={\mathcal P}_m(\C^n)$. 
If $S$ is a compact convex set in $\R^n_+$
with $0\in S$ then the space $\mathcal{P}^S(\C^n)$ forms a 
polynomial ring where the degree of a polynomial $p$
is the minimal $m\in \N$ such that $p\in \mathcal{P}^S_m(\C^n)$. 

The polynomial spaces
${\mathcal P}^S_m(\C^n)$  
appear in  Shiffman-Zelditch \cite{ShiZel:2004}
with $S$ as an  integral polytope and later in Bayraktar \cite{Bay:2017} as
sparse polynomials.   These polynomial classes 
and their relation to pluripotential theory 
have been studied by Bloom, Levenberg and their collaborators 
\cite{Bay:2015,Bay:2016,Bay:2017a,Bay:2017,Bay:2020, 
BayBloLev:2018,BayBloLevLu:2019,
BayHusLevPer:2020,
BloLev:2003,BloLev:2015, 
BloShi:2015,
BosLev:2018,LevPer:2020,Per:2023}.  
Unfortunately, in some of these papers there are false results stated
that we have not seen corrected.  We point them out and
correct them as far as we can.

\smallskip
In Section \ref{sec:02} 
we define the Siciak function  with respect to $S$, $E$, and
$q$ for  every function $q\colon E\to \R\cup\{+\infty\}$
on a subset $E$ of $\C^n$ by 
$\Phi^S_{E,q}=\varlimsup_{m\to\infty}\Phi^S_{E,q,m}$, where
$$
\Phi^S_{E,q,m}=\sup\{|p|^{1/m}\,;\, p\in {\mathcal P}^S_m(\C^n),
\|pe^{-mq}\|_E\leq 1\}, \qquad m=1,2,3,\dots.
$$ 
We drop $S$ in 
the superscript if $S=\Sigma$ and $q$ in the subscript if $q=0$.
The grading of the polynomial classes ${\mathcal P}^S_m(\C^n)$ 
implies  that for every $z\in \C^n$ the sequence $(-\log
\big(\Phi^S_{E,q,m}(z)\big)^{m})_{m\in \N^*}$ is
subadditive and a lemma by Fekete 
implies that 
$$
\Phi^S_{E,q}=\lim_{m\to \infty}\Phi^S_{E,q,m}
=\sup_{m\in \N^*}\Phi^S_{E,q,m}
$$ 
without any restriction on $S$ or $q$. 
This was first proved by
Siciak \cite[Theorem 6.1]{Sic:1962} for $S=\Sigma$.  
For the 
reader's convenience we prove the Fekete Lemma \ref{lem:2.3},
because ingredients from the proof are needed in the proof of 
Proposition \ref{prop:2.2}, where it is shown that the convergence 
is uniform on a compact subset $X$, 
if $q$ is bounded below   on $E$ and $\Phi^S_{E,q}$  
is continuous on $X$.
This statement on uniform convergence
appears in many arguments in pluri\-potential theory.
See for  example Bloom and Shiffman \cite[Lemma 3.2]{BloShi:2015},
Bayraktar \cite[Theorem 2.10]{Bay:2015,Bay:2017} and
Bayraktar, Hussung, Levenberg and Perera 
\cite[Theorem 1.1]{BayHusLevPer:2020}.  
The statements on uniform convergence in these papers 
all follow directly from Proposition \ref{prop:2.2}
under the conditions that $S$ is compact and convex with 
$0\in S$ and $q$ is bounded below.

\smallskip
In Section \ref{sec:03} 
we introduce the {\it logarithmic supporting function} $H_S$
of $S$  and the {\it Lelong class with respect to $S$}.
The function 
$H_S$ is defined on $\C^{*n}$ as the supporting function $\varphi_S$
of $S$, $\varphi_S(\xi)=\sup_{s\in S}\scalar s\xi$, $\xi\in \R^n$,
in logarithmic coordinates extended to $\C^n\setminus \C^{*n}$ as an
upper semicontinuous function and the  Lelong class  $\L^S(\C^n)$ with
respect to $S$ is  defined as the set of all  $u\in \PSH(\C^n)$ 
satisfying $u\leq c_u+H_S$, where $c_u$ is a  constant only depending
on $u$ and $\PSH(X)$ denotes the set of plurisubharmonic functions on 
an open set $X\subset \C^n$.  
We have $H_\Sigma(z)=\log^+\|z\|_\infty$, which implies that
$\L^\Sigma(\C^n)$ is equal to the Lelong class
$\L(\C^n)$.

\smallskip
What values $H_S$ takes at points on the coordinate hyperplanes
$\C^{n}\setminus \C^{*n}$  is opaque from its definition itself, 
but Proposition \ref{prop:3.3}  provides an explicit formula.
Additionally,  this formula shows that the zero set
${\mathcal N}(H_S)$ of $H_S$ may very well be unbounded.  
We prove in Proposition \ref{prop:3.4} that $H_S$ extends  from
$\C^{*n}$ as a continuous plurisubharmonic function on $\C^n$.

\smallskip
The usual grading of the polynomials can be characterized by growth, and
so is also the case for $\mathcal{P}^S(\C^n)$.   
It is clear that if $p\in \mathcal{P}^S_m(\C^n)$
then $|p|\leq C_pe^{H_S}$ for some constant $C_p>0$. In
Theorem  \ref{thm:3.6} we prove  a Liouville type theorem 
which states that if $f\in\OO(\C^n)$ 
satisfies a growth estimate 
$|f(z)|\leq C(1+|z|)^ae^{mH_S(z)}$ for some $C>0$ and some
$a>0$ strictly less than the distance from 
$mS$ to $\N^n\setminus mS$ in $L^1$-norm, 
then $f\in {\mathcal P}^S_m(\C^n)$.

\smallskip
In Section \ref{sec:04}  we define  the {\it  Siciak-Zakharyuta 
function} 
$$
V^S_{E,q}=\sup\{u\,;\, u\in \L^S(\C^n),\; u|_E\leq q\}
$$
with respect to $S$ and any $q \colon E\to \R\cup\{+\infty\}$
on a subset $E$ of $\C^n$.  We drop $S$ in 
the superscript if $S=\Sigma$ and $q$ in the subscript if $q=0$.
By Klimek \cite[Example 5.1.1]{Kli:1991}
$V_K(z)=\log^+(\|z-a\|/r)$ for any norm $\|\cdot\|$ 
and $K=\{z\in \C^n\,;\, \|z-a\|\leq r\}$, $r>0$, the closed 
ball in $\|\cdot\|$ with center $a$ and radius $r$.
These classical examples can not be generalized 
for $V^S_{E}$ for the simple reason that 
the Lelong class $\L^S(\C^n)$ with respect to $S$ 
does not need to be translation invariant.  
We prove in  Proposition \ref{prop:4.3} that $V^S_{E}=H_S$ for any
subset $E$ of $\C^n$ such that 
$\T^n\subseteq E\subseteq \mathcal N(H_S)$, where $\T$ denotes the
unit circle  in $\C$ and $\mathcal N(H_S)$ the zero set of $H_S$.   
This is a generalization of
Bayraktar  \cite[Example 2.3]{Bay:2017}. 

\smallskip
We introduce {\it admissible weights} in 
Definition \ref{def:4.4}, where  we follow 
Bloom \cite[Appendix~B]{SaffTotik:1997} with
the natural generalization for the Lelong classes with respect to  $S$ 
given in  Bayraktar, Bloom, and Levenberg \cite{BayBloLev:2018}.
In Proposition \ref{prop:4.5} we prove that 
for every compact convex $S\subset \R^n_+$  with $0\in S$ and
admissible weight $q$ on a closed subset $E$ the upper regularization
$V^{S*}_{E,q}$ of $V^S_{E,q}$ is in $\L^S_+(\C^n)$, the set
of $u\in \L^S(\C^n)$ such that $H_S-c_u\leq u$ for 
some constant $c_u$.  

\smallskip
It is a fundamental problem to characterize those $S$ and 
admissible weights $q$ for which $V^S_{E,q}=\log \Phi^S_{E,q}$. 
The classical Siciak-Zakharyuta theorem states that $V_K=\log \Phi_K$, 
for every compact subset $K$ of $\C^n$, see \cite[Theorem 5.1.7]{Kli:1991}.
It was first proved by Zakharyuta \cite{Zak:1976} and generalized 
to $V_{K,q}=\log \Phi_{K,q}$ for continuous $q$ on a compact $K$  
by Siciak \cite[Theorem 4.12]{Sic:1981}.  
Magnússon, Sigurðardóttir, and Sigurðsson  
\cite{MagSigSig:2023}  prove that 
for every compact and convex $S\subset \R^n_+$  with $0\in S$
and every  admissible weight $q$  on a closed $E\subset \C^n$,
the equation 
$V^S_{E,q}=\log\Phi^S_{E,q}$ holds on $\C^{*n}$
if and only if $S\cap \Q^n$ is dense in $S$. 
In the case when $V^S_{E,q}$ is lower semicontinuous equality holds on $\C^n$. 
In Proposition \ref{prop:4.7} 
we show that if $S\cap\Q^n$ is not dense in $S$, then
$V^S_{E,q}\neq\log \Phi^S_{E,q}$ for any admissible weight $q$ on a
closed set $E\subseteq \C^n$.   
Magnússon, Sigurðsson, and Snorrason \cite{MagSigSno:2023} prove 
a generalization of 
the Bernstein-Walsh-Siciak theorem  to the weighted setting  with
approximation by polynomials from  the polynomial 
ring ${\mathcal P}^S(\C^n)$.

\smallskip
In Section \ref{sec:05} we study regularity of $V^S_{E,q}$.
We introduce a regularization operator
$R_\delta$ that preserves the class $\L^S(\C^n)$ and has the property
that $R_\delta u\in \mathcal C^\infty(\C^{*n})\cap \L^S(\C^n)$  for
every $u\in \L^S(\C^n)$ and $R_\delta u\searrow u$ as $\delta\searrow 0$. 
This property enables us to prove in 
Proposition \ref{prop:5.4} that $V^{S}_{K,q}$ is lower semicontinuous
on $\C^{*n}$ for every $S$ and every admissible weight $q$ on a
compact set $K$. It is crucial to know if 
the upper regularization $V^{S*}_{E,q}$ of $V^{S}_{E,q}$ satisfies 
$V^{S*}_{E,q}\leq q$ on $K$, for then $V^{S*}_{E,q}=V^{S}_{E,q}$ and
$V^{S}_{E,q}$ is continuous at every point where it
is lower semicontinuous.    
    
It is a natural question to ask under which
conditions on $S$ the class $\L^S(\C^n)$ is preserved under
the standard method for 
regularization of plurisubharmonic functions, that is  convolution 
$u\mapsto u*\psi$,
where 
$\psi \in {\mathcal C}_0^\infty(\C^n)$ with
$\psi\geq 0$, and $\int_{\C^n}\psi \, d\lambda=1$.  
In Theorem \ref{thm:5.8} we
prove that $\L^S(\C^n)$ is preserved under convolution if and only if 
$S$ is a lower set, which says that the cube
$C_s=[0,s_1]\times\cdots\times[0,s_n]$ 
is contained in $S$ for every $s\in S$.  This is a  complete answer to
a question raised by 
Bayraktar, Hussung, Levenberg, and Perera \cite[Section 2]{BayHusLevPer:2020}.

\smallskip
In Section \ref{sec:06} we consider an equilibrium measure
$\mu^S_{E,q}=\big(dd^c V^{S*}_{E,q}\big)^n$  for $V^{S*}_{E,q}$. 
We prove in Theorem \ref{thm:6.1} that its 
support is located where $V^{S*}_{K,q}\geq q$ and
in Theorem \ref{thm:6.2} we prove that 
 its total mass is
$\mu^S_{E,q}(\C^n)=(2\pi)^nn!\operatorname{vol}(S)$ where
$\operatorname{vol}$ denotes the euclidean volume. See  also 
Bayraktar \cite[Proposition 2.7]{Bay:2017}
 and Rashkovskii \cite[Section 3]{Ras:2000}.  

\smallskip
In Section \ref{sec:07} we return to the problem of characterizing
the classes ${\mathcal P}^S_m(\C^n)$ by growth properties we started
in Theorem \ref{thm:3.6}.  Our main result, 
Theorem \ref{thm:7.2}, is that
every  entire function $f$ in the weighted $L^2$-space
$L^2(\C^n,\psi)$ consisting of all measurable functions which are
square integrable with respect to the Lebesgue measure $\lambda$ on 
$\C^n$ with weight $e^{-\psi}$ and $\psi=2mH_S+a\log(1+|\cdot|^2)$
is a polynomial in ${\mathcal P}^{\widehat S_\Gamma}_m(\C^n)$,  
where $\widehat S_\Gamma$ is the $\Gamma$-hull of $S$ 
defined by $\widehat S_\Gamma=\{x\in \R^n_+\,;\, \scalar x\xi\leq
\varphi_S(\xi),  \forall \xi\in \Gamma\}$ 
for a certain cone  $\Gamma\subset \R^n$.
Bayraktar, Hussung, Levenberg, and Perera 
\cite[Proposition 4.3]{BayHusLevPer:2020}
claim that $f\in {\mathcal P}^S_m(\C^n)$ for any polytope $S$ and $a$
sufficiently small.  Example \ref{ex:7.4} shows that their claim is 
false even if $S$ is a polytope containing a neighborhood of $0$ in
$\R^n_+$.  The  Siciak-Zakharyuta type theorem
\cite[Theorem 1.1]{BayHusLevPer:2020}
does not have a sound proof as it is based on their Proposition 4.3.
As far as we can see the proof is only valid if $S$ is a lower set.

\smallskip
In Section \ref{sec:08} we continue the discussion in Perera
\cite{Per:2023} and show that for any polynomial map $f\colon 
\C^n\to \C^\ell$ and any compact convex $0\in S\subseteq \C^n$, there
is a canonical minimal choice of $S'$ such that $f^*\colon \mathcal
P^S_m(\C^n)\to \mathcal P^{S'}_m(\C^\ell)$ is well-defined for all
$m\in \N$ and  $f^*\colon \L^S(\C^n)\to \L^{S'}(\C^\ell)$ is 
well-defined. In the case when $\ell=n$ we show 
when such pullbacks are bijective. 

\smallskip
Snorrason \cite{Sno:2024} generalized  
the Siciak product formula, 
Siciak \cite{Sic:1962,Sic:1981} and 
Klimek \cite[Theorem 5.1.8]{Kli:1991}.
He shows in Corollary 1.3, 
that the gen\-er\-al\-iza\-tion of Bos and Levenberg
\cite{BosLev:2018} of Siciak's product formula only holds for 
lower sets.  This shows that both 
Levenberg and Perera \cite[Proposition 1.3]{LevPer:2020}, 
and Nguyen Quang Dieu and Tang Van Long 
\cite[Theorem 1.3]{NguTan:2021} do not hold.  
The error in both the papers is the same, the authors implicitly 
assume that $\varphi_S(\xi)=\varphi_S(\xi^+)$ holds for every
$\xi\in \R^n$, where $\xi_j^+ = \max\{0,\xi_j\}$,
but in Theorem \ref{thm:5.8} we prove that this identity holds if and
only if $S$ is a lower set. 
In \cite[Section 5]{Sno:2024} Snorrason shows that the sublevel sets 
of $V^S_K$ are not convex in general. 
This contradicts     
Nguyen Quang Dieu and Tang Van Long \cite[Theorem 1.2]{NguTan:2021}
where they claim that for every convex body $S$ and every compact 
convex $K\subset \C^n$ the sublevel sets of $V^S_K$ are convex.
All these mistakes showed us the importance of a careful study of the values
of $H_S$ near points on the union of the coordinate hyperplanes as we
have done in Propositions \ref{prop:3.3} and \ref{prop:3.4}.

\subsection*{Acknowledgment}  
The results of this paper are a part of a research project, 
{\it Holomorphic Approximations and Pluripotential Theory},
project grant 
no.~207236-051, supported by the Icelandic Research Fund.
We would like to thank the Fund for its support and the 
Mathematics Division, Science Institute, University of Iceland,
for hosting the project.   Parts of the paper were written while
the second and third authors were visitors at 
Mathe\-mat\-i\-cal Sciences,
Chalmers University of Technology and University of Gothen\-burg,
 Sweden. They would like to thank the
members of the institute for their hospitality and generosity. 
We thank the referees for their constructive and helpful comments 
which have improved the quality and clarity of the paper. 

\section{Weighted polynomial classes and Siciak functions}
\label{sec:02}

\noindent
Let $S$ be a bounded subset of $\R_+^n = \{x\in \R^n; x_j\geq 0 
\text{ for all } j=1,\ldots,n\}$.
For every $m\in \N$ we associate to $S$  the space 
${\mathcal P}_m^S(\C^n)$  of all polynomials 
in $n$ complex variables of the form
\begin{equation}
  \label{eq:2.1}
  p(z)=\sum_{\alpha \in (mS)\cap \N^n} a_\alpha z^\alpha, \qquad
  z\in \C^ n,
\end{equation}
with the standard multi-index notation
and let ${\mathcal P}^S(\C^n)=\bigcup_{m\in \N} {\mathcal P}_m^S(\C^ n)$.  
If $S$ is the standard simplex
$\Sigma =\ch\{0,e_1,\dots,e_n\}$, then 
the space 
${\mathcal P}_m^\Sigma(\C^n)$ consists of all polynomials of degree
$\leq m$, which we denote  by  ${\mathcal P}_m(\C^n)$.
We let  ${\mathcal P}(\C^n)=\bigcup_{m\in \N} {\mathcal P}_m(\C^ n)$
denote the space of all polynomials in $n$ complex variables. 

\smallskip
Assume now that  $S$ is a compact convex subset of $\R_+^n$ with $0\in S$.
If $\alpha \in jS$ and $\beta
\in kS$ for some  $j,k\in \N^*$, 
say $\alpha=ja$ and $\beta=kb$ with $a,b\in S$, then
convexity of $S$ gives 
$ \alpha+\beta =(j+k)\big((1-\lambda)a + \lambda b \big)  \in (j+k)S$, 
where $\lambda=k/(j+k)\in [0,1]$.  Thus, 
 $z^\alpha z^\beta\in {\mathcal P}^S_{j+k}(\C^n)$ and 
by taking linear combinations of products of monomials we get 
\begin{equation}
    \label{eq:2.2}
{\mathcal P}_j^S(\C^n){\mathcal P}_k^S(\C^n)\subseteq {\mathcal
  P}_{j+k}^S(\C^n).
\end{equation}
This tells us  that  ${\mathcal P}^S(\C^n)$ is a subring of
$\mathcal P(\C^n)$.  
For every $p\in {\mathcal P}^S(\C^n)$ we define the {\it $S$-degree}  
$\degree^S(p)$ of $p$ as the infimum over $m$ for which  $p\in
{\mathcal P}_m^S(\C^n)$.  We have 
\begin{align}
  \label{eq:2.3}
  \degree^S(p_1+p_2) &\leq   \max\{\degree^S(p_1),\degree^S(p_2)\},\\ 
  \degree^S(p_1p_2) &\leq   \degree^S(p_1) +  \degree^S(p_2).
  \label{eq:2.4}
\end{align}
Equality does not hold in general  in either of these
inequalities.

\begin{definition}\label{def:2.1}
Let $S\subset \R^n_+$ be a compact convex set,  $0\in S$,
and   $q\colon E\to \R\cup\{+\infty\}$ be a function on 
$E\subset \C^n$.  For  $m\in \N^*=\{1,2,3,\dots\}$  
the {\it $m$-th Siciak extremal function with respect to $S$, $E$, and
  $q$}  is defined by 
\begin{equation*}
  \Phi^S_{E,q,m}(z)
=\sup \{|p(z)|^{1/m} \,;\, p\in {\mathcal P}_{m}^S(\C^n),
\|pe^{-mq}\|_E\leq 1 \}, 
\qquad z\in \C^n,
\end{equation*}
the {\it Siciak extremal function with respect to $S$, $E$, and $q$} is 
defined by
\begin{equation*}
  \Phi^S_{E,q}(z)=\varlimsup_{m\to \infty}  \Phi^S_{E,q,m}(z), 
\qquad z\in \C^n.
\end{equation*} 
We drop $S$ in the superscripts 
if $S=\Sigma$ and $q$ in the subscripts if $q=0$.   
Note that the family $\{p\in \mathcal P_m^S(\C^n); \|pe^{-mq}\|_E \leq 1\}$ is 
never empty since it always contains the zero polynomial. 
Furthermore, we define $\Phi^S_{E,q,0} = 1$.
\end{definition}

\smallskip
Observe that our definition of  $\Phi^S_{E,q,m}$
deviates from the original definition  of 
Siciak \cite{Sic:1962}, which is $\big(\Phi^S_{E,q,m}\big)^m$
in our notation.  
We have that $\Phi^S_{E,q}, \Phi^S_{E,q,m}$ 
are lower semicontinuous on $\C^n$
for $m=1,2,3,\dots$, for all these functions are  
suprema of continuous functions. 

If $q$ is bounded below, say by the real number $q_0$, then
the constant polynomial $p(z)=e^{mq_0}$ is in ${\mathcal P}^S_m(\C^n)$ and
$\|pe^{-mq}\|_E=\|e^{-m(q-q_0)}\|_E\leq 1$. Hence, it follows that
$\Phi^S_{E,q,m}(z)\geq e^{q_0}$ for every $z\in \C^n$ and
$m\in \N^*$.  

\smallskip
\begin{proposition}
  \label{prop:2.2} 
Let $S\subset \R^n_+$ be a compact convex set with  $0\in S$,
$E\subset \C^n$, and   $q\colon E\to \R\cup\{+\infty\}$ be a function.
Then for $j,k=1,2,3,\dots$
\begin{equation}
  \label{eq:2.5}
  \big(\Phi^S_{E,q,j}(z)\big)^j
\big(\Phi^S_{E,q,k}(z)\big)^k \leq 
\big(\Phi^S_{E,q,j+k}(z)\big)^{j+k}, \qquad z\in \C^n,
\end{equation}
and 
\begin{equation}
  \label{eq:2.6}
  \Phi^S_{E,q}(z)= \lim_{m\to \infty} \Phi^S_{E,q,m}(z)
=\sup_{m\geq 1} \Phi^S_{E,q,m}(z), \qquad z\in \C^n.
\end{equation}
If $q$ is bounded below  and  
$\Phi^S_{E,q}$ is continuous on some compact
subset $X$ of $\C^n$, then  the convergence is uniform on $X$.
\end{proposition}

We need ingredients from the proof from 
Tsuji  \cite[Lemma after Theorem III.25 on page 73]{Tsuji:1959}:

\begin{lemma} {\bf (Fekete lemma)} \ 
  \label{lem:2.3}
Let $(a_m)_{m\in \N^*}$ be a subadditive real  sequence, 
that is  $a_{j+k}\leq a_j+a_k$ for $j,k=1,2,3,\dots$.  Then
$\lim\limits_{m\to\infty} a_m/m=\inf\limits_{m\geq 1}a_m/m$.
\end{lemma}

\begin{proof}
Denote the infimum of $a_m/m$ by $\alpha$ and take $\beta\in \R$ such that
$-\infty\leq \alpha <\beta$. Then there exists $j\in \N^*$ such that
$a_j/j<\beta$. Every $m>j$ can be written as $m=sj+r$ for some
$s,r\in \N$ with  $s\geq 1$  and $0\leq r<j$.  By assumption, we have
$a_m\leq a_{sj}+a_r\leq sa_j+a_r$, so
\begin{equation}
  \label{eq:2.7}
\dfrac {a_m}m \leq \dfrac{sa_{j}+a_r}{m} < \beta \big(1-r/m\big)
+\dfrac{\max_{\ell<j}a_\ell}m,  
\end{equation}
which implies 
$\varlimsup\limits_{m\to \infty} a_m/m \leq \beta$.  Since $\beta$ is
arbitrary  we have
$$
\varlimsup_{m\to \infty}a_m/m 
\leq \inf_{m\geq 1} a_m/m 
\leq \varliminf_{m\to \infty}a_m/m. 
$$
Hence, the limit (possibly $-\infty$) exists and the equality holds.
\end{proof}

\begin{prooftx}{Proof of Proposition \ref{prop:2.2}}
Take $p_j\in{\mathcal P}^S_j(\C^n)$ with
$\|p_je^{-jq}\|_E\leq 1$ and
$p_k\in{\mathcal P}^S_k(\C^n)$ with
$\|p_ke^{-kq}\|_E\leq 1$.  
Then $\|p_jp_ke^{-(j+k)q}\|_E\leq 1$ and
(\ref{eq:2.2}) implies 
$p_jp_k\in {\mathcal P}^S_{j+k}(\C^n)$. Hence,
$|p_j(z)p_k(z)|\leq \big(\Phi^S_{E,q,j+k}(z)\big)^{j+k}$.
By taking supremum over $p_j$ and $p_k$ 
(\ref{eq:2.5}) follows. By (\ref{eq:2.5}) 
the sequence defined by $a_m=-\log \big(\Phi^S_{E,q,m}(z)\big)^{m}$ is
subadditive for every $z$, 
so (\ref{eq:2.6}) follows from Lemma \ref{lem:2.3}.

Assume now that $q\geq q_0$ for some $q_0\in \R$ and
that $\Phi^S_{E,q}$ is continuous on $X$. 
By the discussion after Definition \ref{def:2.1} we have
$\Phi^S_{E,q,m}\geq e^{q_0}$.
Since  $\Phi^S_{E,q}=\sup_{m\geq 1}\Phi^S_{E,q,m}$, it follows by a
simple compactness argument that it is sufficient to show that 
for every $z_0\in X$ and every $\varepsilon>0$ there exists 
$\delta>0$ and $k\in \N^*$ both depending on $z_0$ and $\varepsilon$
such that 
\begin{equation}
  \label{eq:2.8}
  \Phi^S_{E,q}(z)-\Phi^S_{E,q,m}(z)<\varepsilon, \qquad z\in 
  B(z_0,\delta)\cap X, \ m\geq k.
\end{equation}
Let $c=\sup_X \Phi^S_{E,q}$ and choose $\gamma>0$ such that 
$c(1-e^{-\gamma})<\tfrac 14 \varepsilon$ and $j\in \N^*$ so large that
$\Phi^S_{E,q}(z_0)-\Phi^S_{E,q,j}(z_0)<\tfrac 14 \varepsilon$.  
Since $\Phi^S_{E,q}$ is continuous on $X$ and
 $\Phi^S_{E,q,j}$ is lower semicontinuous on $\C^n$, there exists
$\delta>0$ such that for all 
$z\in B(z_0,\delta)\cap X$ we have
$$
\Phi^S_{E,q}(z)-\Phi^S_{E,q}(z_0) <\tfrac 14 \varepsilon 
\quad \text{ and } \quad 
\Phi^S_{E,q,j}(z_0)-\Phi^S_{E,q,j}(z) <\tfrac 14 \varepsilon.
$$
The three estimates imply
\begin{equation}
  \label{eq:2.9}
  \Phi^S_{E,q}(z)-\Phi^S_{E,q,j}(z)<\tfrac 34 \varepsilon, \qquad z\in 
  B(z_0,\delta)\cap X, 
\end{equation}
and (\ref{eq:2.8}) follows from (\ref{eq:2.9})
if we can prove that there exists $k>j$ such that 
\begin{equation}
  \label{eq:2.10}
  \Phi^S_{E,q}(z)-\Phi^S_{E,q,m}(z)
\leq   \Phi^S_{E,q}(z)-\Phi^S_{E,q,j}(z)+\tfrac 14 \varepsilon, 
\qquad z\in X, \ m\geq k.   
\end{equation}
For every $z\in \C^n$ the sequence  $a_m=-\log\big(\Phi^S_{E,q,m}(z)\big)^m$ is 
subadditive. By  (\ref{eq:2.7}) we have 
for every $m>j$ written as $m=sj+r$ with $s\in \N^*$ and $r\in \N$
with $0\leq r<j$ that 
$$
-\log\Phi^S_{E,q,m}(z) \leq 
\dfrac{-sj\log\Phi^S_{E,q,j}(z)-r\log\Phi^S_{E,q,r}(z)}m, \qquad z\in \C^n.
$$ 
We have $sj/m=1-r/m$, 
$-\log \Phi^S_{E,q,j}(z)\geq -\log c$ for every $z\in X$, and 
that $\log \Phi^S_{E,q,r}(z)\geq q_0$ for every $z\in \C^n$, so 
\begin{align*}
\log\Phi^S_{E,q,m}(z) &\geq \log\Phi^S_{E,q,j}(z)
+\dfrac{-r\log\Phi^S_{E,q,j}(z)+r\log\Phi^S_{E,q,r}(z)}m\\
&\geq 
\log\Phi^S_{E,q,j}(z)-  \dfrac{j\log(c/e^{{q_0}})}m, \qquad z\in X.
\end{align*}
We choose $k>j$ so large that $j\log(c/e^{{q_0}})/k<\gamma$.
Then  $\Phi^S_{E,q,m}(z)\geq \Phi^S_{E,q,j}(z)e^{-\gamma}$ and 
$$
\Phi^S_{E,q}(z)-\Phi^S_{E,q,m}(z)
\leq \Phi^S_{E,q}(z)-\Phi^S_{E,q,j}(z)+
\Phi^S_{E,q,j}(z)\big(1-e^{-\gamma}\big)
$$
for every $z\in X$ and $m\geq k$. 
Since $\Phi^S_{E,q,j}(z)\leq c$ for every $z\in X$ and
$c(1-e^{-\gamma})<\tfrac 14 \varepsilon$  the estimate 
(\ref{eq:2.10}) holds.
\end{prooftx}

\section{The Lelong class with respect to a convex set}
\label{sec:03}

Let us begin by setting some notation for the sequel. 
We denote by
$\H(X)$, $\SH(X)$, $\PSH(X)$,  $\OO(X)$, $\LSC(X)$, and $\USC(X)$ 
the classes of harmonic and subharmonic functions on a domain $X$ in $\C$,
plurisubharmonic and holomorphic functions on a complex manifold $X$, 
and lower and upper semicontinuous functions on a topological space
$X$, respectively.    We define 
the coordinate-wise logarithm of the modulus, exponential function,
and positive part, by
\begin{gather*}
  \Log \colon \C^{*n}\to \R^n, \quad 
\Log(z)=(\log|z_1|,\dots,\log|z_n|), \qquad   z\in \C^{*n},
\\
\Exp\colon \R^n\to \R^n_+, \quad 
\Exp(\xi)=e^\xi=(e^{\xi_1},\dots,e^{\xi_n}), \qquad \xi\in \R^n,
\\
{}^+\colon \R^n\to \R^n_+, \quad 
\xi^+=(\xi_1^+,\dots,\xi_n^+), \quad \xi_j^+=\max\{\xi_j,0\} \quad \xi\in \R^n.
\end{gather*}
We let $\D$ denote the unit disc and $\T$ the unit 
circle in $\C$. 
The Lelong class $\L(\C^n)$ 
is the set of all $u\in \PSH(\C^n)$ such that
for some constant $c_u$ depending on $u$ we have
$$
u(z)\leq c_u+\log^+\|z\|_\infty, \qquad z\in \C^n.
$$  
It is clear that $\log^+\|\cdot\|_\infty$ can be replaced by 
$\log^+\|\cdot\|$ or $\log(1+\|\cdot\|)$ for any norm $\|\cdot\|$ on
$\C^n$.

\begin{definition}\label{def:3.1}
For every compact subset of $\R^n_+$ with $0\in S$ 
we define the \emph{supporting function of $S$} as 
\begin{equation*}
  \varphi(x) = \sup_{s\in S} \langle s,x\rangle, 
  \qquad x\in \R^n,
\end{equation*}
and the \emph{logarithmic supporting function of $S$}
as the function $H_S\colon \C^n\to \R_+$  defined on $\C^{*n}$ by
  \begin{equation*}
    H_S(z)=(\varphi_S\circ \Log)(z)=
\max_{s\in S}\big(s_1\log|z_1|+\cdots+s_n\log|z_n|\big),
\qquad z\in \C^{*n}, 
  \end{equation*}
and  extended  to  $\C^n\setminus \C^{*n}$ by
  \begin{equation*}
    H_S(z)= \varlimsup_{\C^{*n}\ni w\to z}H_S(w),
\qquad z\in \C^{n}\setminus \C^{*n}. 
  \end{equation*} 
The real number 
$\sigma_S=\varphi_S({\mathbf 1})$, where 
${\mathbf 1}=(1,\dots,1)\in \R_+^n$,
is called the {\it logarithmic type} of $H_S$.
\end{definition}

\smallskip
Since $\varphi_S$ is positively homogeneous of degree $1$
and convex, that is
$\varphi_S(t\xi)=t\varphi_S(\xi)$ 
and $\varphi_S(\xi+\eta)\leq \varphi_S(\xi)+\varphi_S(\eta)$
for every $t\in \R_+$  and $\xi,\eta\in \R^n$, we have 
\begin{gather}   \label{eq:3.1}
  H_S(z)=\tfrac{1}{\lambda}
H_S(|z_1|^\lambda,\dots, |z_n|^\lambda), 
\qquad \lambda\in \R^*_+, \  z\in \C^{*n},\\
 \label{eq:3.2}
  H_S(z_1w_1,\dots,z_nw_n)\leq H_S(z)+H_S(w), \qquad z,w\in \C^{*n}.
\end{gather} 
Observe that $\varphi_S(-{\mathbf 1})=0$ and that
for $z\in\C^{*n}, \ \lambda\in \C^*$
\begin{equation}
  \label{eq:3.3}
  H_S(\lambda z)\leq 
H_S(z)+H_S(|\lambda|{\mathbf 1}) 
=
H_S(z)+ \sigma_S\log^+|\lambda|.
\end{equation}
If we write $z=\|z\|_\infty w$ with
$\|w\|_\infty=1$, then this formula implies
$H_S/\sigma_S \in \L(\C^n)$, 
\begin{equation}
  \label{eq:3.4}
  H_S(z)\leq 
\sigma_S  \log^+\|z\|_\infty,
\qquad  z\in   \C^n.
\end{equation}
Directly from the definition we see that
$H_S\in \PSH(\C^{*n})\cap {\mathcal C}(\C^{*n})$. 
Since $\C^n\setminus\C^{*n}$ is pluripolar and $H_S$ 
is locally bounded above we have
$H_S\in \PSH(\C^n)$. 

\begin{proposition}
    \label{prop:3.2} 
Let $S\subset \R_+^n$ be compact convex and with $0\in S$. 
Then for every $z\in \C^{*n}$ and
$w\in \C^n$ we have 
\begin{equation*}
  H_S(z+w)\leq H_S(z)+
\varphi_S(|w_1|/|z_1|,\dots,|w_n|/|z_n|),
\end{equation*}
and in particular,  for every $w\in \overline \D^n$ and $\delta\in
]0,1[$ we have 
\begin{equation*}
  H_S({\mathbf 1}+\delta w)\leq \delta \sigma_S.
\end{equation*}
Furthermore, 
\begin{equation*}
  H_S(z+w)\leq H_S(z)+\\
\varphi_S(\log^+(|w_1|/|z_1|),\dots,\log^+(|w_n|/|z_n|))+(\log 2) \sigma_S.
\end{equation*}
\end{proposition}

\begin{proof}  By plurisubharmonicity of $H_S$  
and 
  convexity of the functions given by $w\mapsto \varphi_S(|w_1|/|z_1|,\ldots,
  |w_n|/|z_n|)$ we may assume that $w,z+w\in
  \C^{*n}$.  For 
some $t\in S$ we have 
\begin{equation*}
  H_S(z+w)=\log(|z_1+w_1|^{t_1}\cdots|z_n+w_n|^{t_n}) 
\leq \scalar t{\Log \, z}+ \sum_{j=1}^n t_j\log
\Big(1+\frac{|w_j|}{|z_j|}\Big).
\end{equation*}
Since $\log(1+x)\leq x$, for $x\geq 0$ 
the first estimate follows. Since    
${\mathbf 1}+\delta w\in \C^{*n}$ for
$w\in \overline \D^n$, $\delta \in ]0,1[$,  
and $H_S({\mathbf 1})=0$, the second estimate  follows.
For the third estimate we use the fact that
$\log(1+x)\leq \log 2+\log^+ x$ for $x\geq 0$.
\end{proof}

\bigskip
The zero set ${\mathcal N}(H_S)$ of $H_S$
can be understood in terms of the zero set of $\varphi_S$
which is a cone.
Since  ${\mathcal N}(H_S)\cap \C^{*n}=\Log^{-1}({\mathcal N}(\varphi_S))$ 
and $\R^n_-\subseteq {\mathcal N}(\varphi_S)$, the
closed unit polydisc 
$\overline\D^n$ is contained in ${\mathcal N}(H_S)$. Furthermore,
${\mathcal N}(H_S)$ is equal to  
$\overline \D^n$  if and only if  
$\R_+S=\R_+^n$.  
We have a complete description of the values of $H_S$  at every point
in $\C^n\setminus \C^{*n}$, the union of the coordinate hyperplanes. 

\begin{proposition} \label{prop:3.3}
 Let $S$ be a compact convex subset of $\R_+^n$ 
with $0\in S$.  For every $a\neq 0$ in 
some coordinate hyperplane we have 
\begin{equation*}
  H_S(a)=H_{S_J}(a_{j_1},\dots,a_{j_\ell}),
\end{equation*}
where $J\subset \{1,\dots,n\}$ consists of the indices $j_1<\cdots<j_\ell$  
of the  non-zero coordinates $a_{j_1},\dots,a_{j_\ell}$ of  $a$
and  $S_J\subseteq \R^\ell$  consists of all $t\in \R^\ell$ such that if
$s\in \R^n_+$ is defined by 
$s_{j_k}=t_k$ for $j_k\in J$ and
$s_j=0$ for $j\not\in J$, then $s\in S$.
\end{proposition}

  \begin{proof}  
After renumbering the variables  we may assume that
$J=\{1,\dots,\ell\}$     
and  $a_{\ell+1}=\cdots=a_n=0$.  We write $z=(z',z'')\in \C^n$, where 
$z'\in \C^\ell$ and $z''\in \C^{n-\ell}$.  Then $a'\in \C^{*\ell}$, so 
$H_{S_J}$ is continuous at $a'$ and we have 
\begin{align*}
  H_{S_J}(a')&=\lim_{\C^{*\ell}\ni w'\to a'}H_{S_J}(w')
=\lim_{\C^{*m}\ni w'\to a'}\sup_{t\in S_J}\scalar t{\Log\,  w'}\\
&=\lim_{\C^{*m}\ni w'\to a'}
\sup_{s=(t,0)\in S}\scalar s{(\Log\,  w',0,\dots,0)}\nonumber\\
&\leq \varlimsup_{\C^{*n}\ni w\to a}
\sup_{s\in S}\scalar s{\Log\,  w}=H_S(a).
\nonumber
\end{align*}
In order to prove the converse inequality we take a 
convergent
sequence in $\C^{*n}$,
 $w_j=(w_{j,1},\dots,w_{j,n})\to a$ such that
$\lim_{j\to \infty}H_S(w_j)=H_S(a)$.
There exists $s_j\in S$ such that
$H_S(w_j)=\scalar{s_j}{\Log\,  w_j}$ for every $j$. Since $H_S\geq 0$ and
$\log|w_{j,\kappa}|\to -\infty$ as $j\to\infty$ for $\kappa=\ell+1,\dots,n$, 
it follows that $s_{j,\kappa}\to 0$ as $j\to\infty$ for $\kappa=\ell+1,\dots,n$.
By compactness of $S$  there exists a subsequence
$s_{j_k}$ converging to $(t,0)\in S$.  We have $t\in S_J$ and conclude
that
\begin{align*}
    H_S(a)= \lim_{k\to \infty}H_S(w_{j_k})=
\lim_{k\to \infty}\scalar{s_{j_k}}{\Log\, w_{j_k}}=
\scalar t{\Log\,  a'} \leq H_{S_J}(a').
\end{align*}
\end{proof}

\smallskip
With a proof similar to
Rashkovskii  \cite[Proposition 2.2]{Ras:2001}
we are able to show that $H_S\in {\mathcal C}(\C^n)$:

\begin{proposition}\label{prop:3.4}
Let $S$ be a compact convex subset of $\R_+^n$ 
with $0\in S$.  
Then $H_S$ is plurisubharmonic and continuous on $\C^n$. 
\end{proposition}

\begin{proof}
Let $0 \neq a \in \C^n\setminus \C^{*n}$.
Since $H_S\in \PSH(\C^n)\cap {\mathcal C}(\C^{*n})$,  
it suffices to prove that $H_S$ is lower semicontinuous at $a$.
After renumbering the variables we may assume
$a_{\ell+1}=\dots = a_n = 0$ and $a_j \neq 0$ for $j \leq \ell$.
We also write $z = (z', z'') \in \C^n$, where $z' \in \C^\ell$ and $z'' \in \C^{n - \ell}$.
Let $z_j \in \C^n$ be such that $z_j \rightarrow a$.
Since $H_S$ is rotationally invariant in each variable it takes the
constant value $H_S(z_j)$ on the distinguished boundary 
of the $n-\ell$ dimensional polydisc
$\{(z',\zeta'') \,;\,  |\zeta''_k|\leq |z''_{j,k}|,
k=1,\dots,n-\ell\}$, so by the maximum principle
$H_S(z_j', 0)\leq  H_S(z_j)$.
By Proposition \ref{prop:3.3} we have that $H_S(z_j', 0) = H_{S_J}(z_j')$,
    for $J = \{1, \dots, \ell\}$.
By the continuity of $H_{S_J}$ at $a'$ we have
\begin{equation*}
    \varliminf_{j \rightarrow \infty} H_S(z_j)
    \geq
    \varliminf_{j \rightarrow \infty} H_S(z_j', 0)
    =
    \varliminf_{j \rightarrow \infty} H_{S_J}(z_j')
    =
    H_{S_J}(a')
    =
    H_S(a).
\end{equation*}
This proves the lower semicontinuity of $H_S$ at $a$.
\end{proof}

\begin{definition}\label{def:3.3}
For every compact convex subset  $S$ of $\R^n_+$ with
$0\in S$  we define the {\it $S$-Lelong class} $\L^S(\C^n)$
as the set of all $u\in \PSH(\C^n)$ such that
\begin{equation*}
  u(z)\leq c_u+H_S(z), \qquad z\in \C^n,
\end{equation*}
for some constant $c_u$ depending on $u$,  and define
$\L^S_+(\C^n)$ as the subclass of functions $u$, that have the
same asymptotic behavior at infinity as the function
$H_S$, that is 
\begin{equation}
-c_u+H_S(z) \leq  u(z)\leq c_u+H_S(z), \qquad z\in \C^n.
\end{equation}  
\end{definition}

The Liouville theorem tells us that an entire function 
$f\in \OO(\C^n)$ which satisfies a growth estimate
$|f(z)|\leq C(1+|z|)^{a+m}$, $z\in \C^n$,  
for some $m\in \N$ and $a\in [0,1[$,
is a polynomial of degree $\leq m$, that is $f\in {\mathcal P}_m(\C^n)$.
The following is a Liouville type theorem for the polynomial classes
${\mathcal P}^S_m(\C^n)$:

\begin{theorem}\label{thm:3.6} Let $d_m$ denote
the distance between $mS$ and $\N^n\setminus mS$ in the $L^1$-norm.  
Then for every $f\in \OO(\C^n)$ the following are equivalent:
\begin{enumerate}
\item [{\bf (i)}] $f \in{\mathcal P}^S_m(\C^n)$.
\item [{\bf (ii)}] $\log |f|^{1/m}\in \L^S(\C^n)$.
\item [{\bf (iii)}]
there exists $a\in[0,d_m[$ such that  
$|f|e^{-mH_S-a\log^+\|\cdot\|_\infty}\in L^\infty(\C^n)$.
\item [{\bf (iv)}] there exists $a\in[0,d_m[$ 
and a constant $C>0$ such that
$$
|f(z)|\leq C(1+|z|)^ae^{mH_S(z)}, \qquad z\in \C^n.
$$
\end{enumerate}
\end{theorem}

\begin{proof} 
{\bf (i)}$\Rightarrow${\bf (ii):}
If $\alpha\in mS$, then 
$|z^\alpha|\leq e^{mH_S(z)}$, so for  
$f \in{\mathcal P}^S_m(\C^n)$, 
$f(z)=\sum_{\alpha\in (mS)\cap
  \N^n}a_\alpha z^\alpha$, we have
$\log |f(z)|^{1/m}\leq c_f/m+ H_S(z)$ with 
$c_f=\log \sum_{\alpha\in (mS)\cap \N^n}|a_\alpha|$. 

\smallskip\noindent
{\bf (ii)}$\Rightarrow${\bf (iii):} We have 
$|f|e^{-mH_S-a\log^+{\|\cdot\|_\infty}}\leq |f|e^{-mH_S}\in L^\infty(\C^n)$. 

\smallskip\noindent
{\bf (iii)}$\Rightarrow${\bf (i):} 
Observe first that $m\varphi_S(\xi)+a\|\xi^+\|_\infty
=\varphi_{mS+a\Sigma}(\xi)$ for every $\xi\in \R^n$. 
Let $f(z)=\sum_{\alpha\in \N^n} a_\alpha z^\alpha$ be
the power series expansion of $f$ at $0$.
We need to show that $a_\alpha=0$ for all $\alpha\in \N^n\setminus mS$. 
Since $a<d_m$ we have $(mS+a\Sigma)\cap \N^n=(mS)\cap \N^n$, 
so $\alpha\in \N^n\setminus (mS+a\Sigma)$.
Hence, there exists $\xi \in \R^n$ such that 
$\scalar \alpha\xi >\varphi_{mS+a\Sigma}(\xi)$.  
We let $C_t$ denote the polycircle with center $0$ and
polyradius $(e^{t\xi_1},\dots,e^{t\xi_n})$ and observe that by
the Cauchy formula for derivatives we have
\begin{equation*}
a_\alpha=\dfrac 1{(2\pi i)^n} \int_{C_t}
\dfrac{f(\zeta)}{\zeta^\alpha} \, 
\dfrac{d\zeta_1\cdots d\zeta_n}{\zeta_1\cdots \zeta_n}. 
\end{equation*}
For $\zeta=(e^{t\xi_1+i\theta_1},\dots,e^{t\xi_n+i\theta_n})$ 
on $C_t$  we have $$|f(\zeta)|/|\zeta^\alpha|\leq
Ce^{-t(\scalar \alpha\xi-\varphi_{mS+a\Sigma}(\xi))},$$ so the right-hand side
tends to $0$ as $t\to +\infty$, and we conclude that  $a_\alpha=0$.

\smallskip\noindent
{\bf (iii)}$\Leftrightarrow${\bf (iv):} 
Follows from the equivalence of the euclidean norm
$|\cdot|$ and $\|\cdot\|_\infty$.
\end{proof}

\medskip
Recall from Klimek \cite[p.~87]{Kli:1991} that a real valued 
$u\in \PSH(X)$ on an open subset $X$ of $\C^n$ is said to be
{\it maximal}, if for every relatively compact open subset $G$ 
of $X$ and every  $v\in \USC(\overline G)\cap \PSH(G)$ satisfying
$v\leq u$ on $\partial G$ we have $v\leq u$ on $G$.
For the reader's convenience we prove the following well known result.

\begin{lemma}\label{lem:3.6}
Let $X$ be an open subset of $\C^n$ and $u\in \PSH(X)$ be real valued.
Assume that for every relatively compact open subset $G$ of $X$ 
there exists a family $(g_z)_{z\in G}$ of holomorphic maps $g_z\colon
D_z\to \C^n$ defined on open subsets $D_z$ of $\C$, such that
$K_z=g_z^{-1}(\overline G)$ is compact,  $z=g_z(\tau_z)$ for some 
$\tau_z\in K_z$, and $u\circ g_z$ is harmonic on $D_z$.  
Then $u$ is maximal on $X$.    
\end{lemma}

\begin{proof} By Klimek \cite[Proposition 3.1.1]{Kli:1991}
we may  take $v\in\PSH(X)$ in the definition of maximality. 
We assume that $v\leq u$ on $\partial G$ and need to
prove that $v(z)\leq u(z)$ for every $z\in G$.
The set $K_z=g^{-1}_z(\overline G)$ is compact, the
function $s_z=v\circ g_z\in \SH(D_z)$ is less than or equal
to $h_z=u\circ g_z\in \H(D_z)$ on the boundary of $K_z$.
By the maximum principle for harmonic functions
$v(z)=s_z(\tau_z)\leq h_z(\tau_z)=u(z)$. 
\end{proof}

\medskip
For every  $z\in \C^{n*}$  we define a
parametric curve $f_z\colon \C\to \C^n$,
$f_z=(f_{z,1},\dots,f_{z,n})$, by
\begin{equation}
  \label{eq:3.6}
  f_{z,j}(\tau)=
  \begin{cases}
e^{-i\tau\log|z_j|}(z_j/|z_j|), &z_j\neq 0,\\
0, &z_j=0,     
  \end{cases}
 \qquad \tau\in \C.
\end{equation} 
We have $f_z(i)=z$ and $\|f_z(\tau)\|_\infty=1$ for every $\tau\in \R$.
If  $\|z\|_\infty>1$,  then for $j$ with $|z_j|=\|z\|_\infty$ we have
\begin{equation}
  \label{eq:3.7}
  f_{z,j}'(\tau)=-ie^{-i\tau\log|z_j|}(\log|z_j|)(z_j/|z_j|)\neq 0.
\qquad \tau\in \C.
\end{equation} 
Hence  $f_z$ parametrizes an open Riemann surface
in $\C^n$ through the point $z$.
Furthermore,  we have $H_S(f_z(\tau))=\Im\tau\, H_S(z)$ for 
$\Im \tau\geq 0$.  The function 
$\C\ni \tau \mapsto H_S(f_z(\tau))$ is subharmonic, 
harmonic in the upper half plane, equal to  $0$ 
on the real axis,  and takes the value $H_S(z)$ at $i$.    
By Lemma \ref{lem:3.6} 
with $f_z$ in the role of $g_z$ and $H_S$ in the role of $u$ 
we get:

\begin{proposition}\label{prop:3.7}  
Let $S$ be  a compact convex subset of $\R_+^n$ with $0\in S$.
Then $H_S$ is maximal on $\C^n\setminus \partial {\mathcal N}(H_S)$, 
where ${\mathcal N}(H_S)$ is the zero set of $H_S$. 
\end{proposition}

\section{Weighted Siciak-Zakharyuta functions}
\label{sec:04}

\begin{definition}\label{def:4.1}
Let $S\subset \R^n_+$ be a compact convex set,  $0\in S$,
and   $q\colon E\to \R\cup\{+\infty\}$ be a function on 
$E\subset \C^n$. The {\it Siciak-Zakharyuta function 
with respect to $S$, $E$,  and
  $q$}  is defined by
\begin{equation*}
  V^S_{E,q}(z)=\sup\{u(z) \,;\, u \in \L^S(\C^n),\; u|_E\leq q\}, 
\qquad z\in \C^n.
\end{equation*}
We drop $S$ in superscripts if $S=\Sigma$ and $q$ in subscripts if $q=0$.   
\end{definition}
 
From Theorem \ref{thm:3.6}
it follows that $\log\Phi^S_{E,q}\leq V^S_{E,q}$ for every 
compact convex $S\subset \R^n_+$ and every  $q\colon E\to
\R\cup\{+\infty\}$ on $E\subset \C^n$. 
We will need a variant of the  Phragmén-Lindelöf principle,
see  \cite[Lemma 2.1]{HorSig:1998}.

\begin{lemma}  \label{lem:4.2}
Let $v$ be subharmonic in the upper half
  plane $\C_+=\{z\in \C\,;\, \Im z>0\}$
such that for some real constants $C$ and $A$ we have
$v(z)\leq C+A|z|$ for all $z\in \C_+$, 
and  $\varlimsup\limits_{\C_+\ni z\to x}v(z) \leq 0$ for all $x\in \R$.
Then $v(z)\leq A\, \Im z$ for all $z\in \C_+$. 
\end{lemma}

By Klimek \cite[Example 5.1.1]{Kli:1991}
$V_K(z)=\log^+(\|z-a\|/r)$ if $\|\cdot\|$ is any norm and
$K=\{z\in \C^n\,;\, \|z-a\|\leq r\}$, $r>0$, is the closed 
ball in this norm with center $a$ and radius $r$.
The polynomial classes ${\mathcal P}^S_m(\C^n)$ are in general not
translation invariant,  so we can not expect to have a generalization
of this example.  
The following is proved in special cases by Bos and Levenberg 
\cite{BosLev:2018}:

\begin{proposition}\label{prop:4.3}
Let $S$ be a compact convex subset of $\R^n_+$ with $0\in S$ and let
$E$ be a subset of $\C^n$ such that $\T^n \subseteq
E\subseteq {\mathcal N}(H_S)$.  Then $V^S_E=H_S$.
\end{proposition}

\begin{proof}  Since $H_S\in \L^S(\C^n)$ and $H_S|_E=0$ we have
  $H_S\leq V^S_E$, so it is sufficient to prove that if
$u\in \L^S(\C^n)$ with $u|_E\leq 0$ we have $u(z)\leq H_S(z)$ for every
$z\in \C^{*n}$ such that
$H_S(z)>0$.  Define $f_z$ by
(\ref{eq:3.6}) and  $v\in \SH(\C)$ by 
$v(\tau)=u(f_z(\tau))$, $\tau\in \C$.  
Since $u\in \L^S(\C^n)$ we have
$v(\tau)=u(f_z(\tau))\leq c_u+H_S(f_z(\tau))=c_u+\Im \tau\, H_S(z)$,
$\Im \tau\geq 0$,
and since $f_z(\R)\subseteq \overline \D^n$ we have
$v\leq 0$ on $\R$. Lemma \ref{lem:4.2} gives 
$u(z)=v(i)\leq H_S(z)$. 
\end{proof}

\medskip
Since $\T^n\subset {\mathcal N}(H_S)$ for every $S$ the maximum
principle implies that 
$\overline \D^n\subset {\mathcal N}(H_S)$, where $\D$ denotes the 
unit disc in $\C$.

\begin{definition}\label{def:4.4}  
Let $0\in S\subset \R^n_+$ be compact and convex and   
 $q\colon E\to \R\cup\{+\infty\}$ be a function on $E
 \subset \C^n$.  We say that $q$  is an 
{\it admissible weight  with respect to $S$
on $E$} if 
\begin{enumerate}
\item[{\bf(i)}]  $q$ is lower semicontinuous,   
\item[{\bf(ii)}] the set $\{z\in E \,;\, q(z)<+\infty\}$ is non-pluripolar,
  and 
\item[{\bf(iii)}] if $E$ is unbounded, then  
$\lim\limits_{E\ni z, |z|\to \infty}(H_S(z)-q(z))=-\infty$.
\end{enumerate}  
\end{definition}

This definition is taken from 
Bloom \cite[Appendix~B: Definition 2.1]{SaffTotik:1997}.
Some authors use the term {\it admissible external field} for $q$ 
rather than {\it weight} in this situation and then they
refer to $e^{-q}$ as a weight. 
Observe that if $q$ is an admissible weight, then $E$ is
non-pluripolar and that if $E$ is unbounded then $q=0$ is not 
an admissible weight.

\begin{proposition}
  \label{prop:4.5}  
Let $S$ be a compact convex subset of $\R^n_+$ with $0\in S$ and 
$q$ be an admissible weight with respect to $S$ on a compact subset 
$K$ of $\C^n$.  Then $V^{S*}_{K,q}\in \L^S_+(\C^n)$.
\end{proposition}

\begin{proof}  The upper regularization
$V^{S*}_{E,q}$ of $V^{S}_{E,q}$ is plurisubharmonic 
in $\C^n$ if $q$ is an admissible weight with respect to $S$ 
on  $E$. 
Since $V^S_{K,q}\leq q$ on $K$, $q$ is admissible on $K$ and
$\{z\in K\,;\, q(z)<+\infty\} 
\subseteq \{z\in \C^n \,;\, V^S_{K,q}(z)<+\infty\}$, 
the set on the right is non-pluripolar. By Klimek \cite[Proposition 5.2.1]{Kli:1991}
it follows that the family $\mathcal{U}=\{u\in \L^S(\C^n)\,;\,
u|_K\leq q\}\subset \sigma_S \L(\C^n)$ is locally
uniformly bounded above, where 
$\sigma_S=\varphi_S({\mathbf 1})$ is the logarithmic type of $H_S$. 
Let $c>0$ be such that $u|_{{\overline{\D}}^n}\leq c$ for 
all $u\in \mathcal{U}$.  Then by Proposition \ref{prop:4.3}, 
we have $V^{S}_{K,q}\leq V^{S}_{{\overline{\D}}^n}+c= H_S+c$. 
Hence, $V^{S*}_{K,q}\in \L^S(\C^n)$. If  
$c=\max_{w\in K}H_S(w)-\min_{w\in K}q(w)$ then $H_S-c\leq
V^{S}_{K,q}$, so $V^{S*}_{K,q}\in \L^S_+(\C^n).$
\end{proof}

\bigskip
Admissible weights on unbounded closed sets yield the same Siciak and
Siciak-Zakharyuta functions as some compact subsets. Admissible
weights do not obstruct the growth of Siciak-Zakharyuta functions.   
See  Bloom \cite[Appendix~B]{SaffTotik:1997}.

\begin{proposition}\label{prop:4.6}
Let $S$ be a compact convex subset of $\R^n_+$ with $0\in S$ and
$q$ an admissible weight on a closed subset $E$ of $\C^n$,
and set $E_R=E\cap \overline B(0,R)$ for every $R>0$.  
Then $V^S_{E,q}=V^S_{E_R,q}$ and $\Phi^S_{E,q}=\Phi^S_{E_R,q}$ 
for $R$ sufficiently large. Furthermore, $V^{S*}_{E,q}\in \L^S_+(\C^n)$.
\end{proposition}

\begin{proof}   Since $E_R\subseteq E$ for every $R>0$ we have
$\Phi^{S}_{E,q}\leq \Phi^{S}_{E_R,q}$ and
$V^S_{E,q} \leq V^S_{E_R,q}$, so we need to prove the reverse 
inequalities.  
By  Definition \ref{def:4.4},  
$E$ is non-pluripolar and by \cite[Corollary 4.7.7]{Kli:1991}
a countable union of  pluripolar sets is pluripolar. It follows 
that  $E_{R_0}$ is non-pluripolar for some $R_0>0$ and consequently
$E_R$  is pluripolar for every $R\geq R_0$.  By 
Proposition \ref{prop:4.5},  $V^{S*}_{E_R,q}\in \L^S(\C^n)$  
and it follows that  for some constant $c>0$
\begin{equation}
  \label{eq:4.1}
  V^{S*}_{E_R,q}(z) \leq c+H_S(z), \qquad z\in \C^n.
\end{equation}
Condition (iii) in Definition \ref{def:4.4} implies that 
we can choose $R_1\geq R_0$ such that 
\begin{equation}
  \label{eq:4.2}
  q(z)-H_S(z) \geq  c, \qquad z\in E\setminus E_{R_1}.
\end{equation}
Now we take $u\in \L^S(\C^n)$ and assume that $u\leq q$ on $E_{R_1}$.
By (\ref{eq:4.1}) we have $u\leq c+H_S(z)$ for all $z\in \C^n$ and by
(\ref{eq:4.2}) we have $u\leq q$ on $E\setminus E_{R_1}$. 
Hence, $u\leq V^S_{E,q}$.

\medskip
Let $p\in \mathcal{P}^S_m(\C^n)$ for some $m\in \N$ be such that
$\|pe^{-mq}\|_{E_{R_1}}\leq 1$. By Theorem \ref{thm:3.6}  we have
$u=\log|p|^{1/m} \in \L^S(\C^n)$ and $u \leq q$ on $E_{R_1}$. Again by
(\ref{eq:4.1}) and 
(\ref{eq:4.2}) we have $u\leq q$ on $E\setminus E_{R_1}$ as
well. Then $\|pe^{-mq}\|_{E}\leq 1$ and  
$|p|^{1/m}\leq \Phi^S_{E,q}$.
The last statement follows directly from Proposition \ref{prop:4.5}.
\end{proof}

\medskip
For a general compact convex $S$ with $0\in S$, 
let  $S'=\overline{S\bigcap \Q^n}$  be the closure  of the set of 
rational points  in $S$. Then
${\mathcal P}^{S}_m(\C^n)={\mathcal P}^{S'}_{m}(\C^n)$  
for every $m\in \N^*$ and consequently 
$\log \Phi^{S}_{E,q}=\log \Phi^{S'}_{E,q}
\leq V^{S'}_{E,q}\leq V^S_{E,q}$.  
Observe that $S=S'$ if $S$ is a convex body but $S \neq S'$ 
for example if $S$ is a line segment in $\R^2$ with 
irrational slope.

\begin{proposition}\label{prop:4.7} 
Let $S\subset \R^n_+$ be compact and convex with $0\in S$. 
If $S\cap \Q^n$ is not dense in $S$,  then for every admissible
weight $q$ on a closed $E\subset \C^n$ we have 
$\log \Phi^S_{E,q}\neq V^S_{E,q}$. 
\end{proposition}

\begin{proof} 
Since  $S'=\overline{S\cap \Q^n}\subsetneq S$, there exists
a $\xi\in \R^n$ with $\varphi_{S'}(\xi)<\varphi_S(\xi)$.  
By Proposition \ref{prop:4.5}
there exist constants  $c$ and $c'$  
such that 
$$
H_{S}(z) - c  \leq V^{S}_{E,q} \quad \text{ and } \quad 
V^{S'}_{E,q}(z)\leq H_{S'}(z)+ c'.
$$ 
For $r>0$ sufficiently large 
we have $\varphi_{S'}(r\xi)+c' <\varphi_S(r\xi)-c$, so for
$z=(e^{r\xi_1}, \ldots, e^{r\xi_n})$
$$\log \Phi^S_{E,q}(z)=\log \Phi^{S'}_{E,q}(z)
\leq V^{S'}_{E,q}(z)\leq H_{S'}(z)+ c' <  H_{S}(z) - c  \leq
V^{S}_{E,q}(z).
$$ 
\end{proof}

The next result regards the Siciak-Zakharyuta functions with respect
to $S$, $E$ and $q$ when we have decreasing sequences of sets
$S_j\searrow S$ or increasing sequences of weights $q_j \nearrow q$.

\begin{proposition}   \label{prop:4.8} 
Let $S_j$, $j\in \N$ and $S$ be compact convex  subsets of 
$\R^n_+$  with $0\in S$ and  $S_j\searrow S$,  
$q$ be an admissible weight on a compact subset $K$ of $\C^n$,
$(K_j)_{j\in\N}$ be a decreasing sequence of compact sets with 
$\bigcap_jK_j=K$, 
and  $(q_j)_{j\in \N}$ be  a sequence in $\LSC(K_j)$
such that $q_j\nearrow q$. Then: 
\begin{enumerate}
\item [{\bf (i)}]  $\L^S(\C^ n)=\bigcap_{j\in \N}\L^{S_j}(\C^ n)$.
\item [{\bf (ii)}]  If $V^{S_j*}_{K,q}\leq q$ on $K$ for some  $j$, then  
$V^{S_j}_{K,q}\searrow V^{S}_{K,q}$  as $j\to \infty$. 
\item[{\bf (iii)}]  Every $q_j$ is an admissible weight on $K_j$, 
$V^{S}_{K_j,q_j}\nearrow V^S_{K,q}$ and $\Phi^S_{K_j,q_j}\nearrow
\Phi^S_{K,q}$ as $j\to \infty$.
\end{enumerate}
\end{proposition}

\begin{proof}
{\bf (i)} \ Obviously $\L^S(\C^n)\subseteq \bigcap_j\L^{S_j}(\C^n)$.
Let $u\in \bigcap_j\L^{S_j}(\C^n)$ and set $c_u=\sup_{{\overline \D}^n}u$.  
Then by Proposition
\ref{prop:4.3} we have $u-c_u\leq V^{S_j}_{\overline{\D}^n}=H_{S_j}$ for every $j$.
We have $H_{S_j}\searrow H_S$, so $u \leq c_u+H_S$ and $u\in \L^S(\C^n)$.

\smallskip\noindent
{\bf (ii)}  We  have  
$V_{K,q}^S\leq V^{S_j}_{K,q}$ and since $V^{S_j*}_{K,q}\leq q$
the equation $V^{S_j*}_{K,q}=V^{S_j}_{K,q}$ holds. 
By  Proposition \ref{prop:4.5} we have
$V^{S_j}_{K,q} \in \L^{S_j}(\C^n)$. Since the sequence is decreasing 
we have $V=\lim_{j\to \infty}V^{S_j*}_{K,q} \in \bigcap_j \L^{S_j}(\C^n)=\L^S(\C^n)$
and $V\leq q$ on $K$.  
Hence, $$ V^S_{K,q}\leq \lim_{j\to \infty}V^{S_j}_{K,q}\leq V\leq V^S_{K,q}.
$$
{\bf (iii)} Since $F=\{z\in K\,;\, q(z)<+\infty\}
\subseteq F_j=\{z\in K_j\,;\, q_j(z)<+\infty\}$ 
for every $j$ and $F$ is non-pluripolar, the set
$F_j$ in non-pluripolar and 
every $q_j$ is an admissible  weight on $K_j$.  
Since the sequence  $(V^{S}_{K_j,q_j})_{j\in \N}$ is increasing 
and bounded above by $V^{S}_{K,q}$ we need to show that 
$V^{S}_{K,q}\leq V=\lim_{j\to \infty}V^{S}_{K_j,q_j}$.
For that purpose we take $u\in \L^S(\C^n)$  with $u\leq q$ on $K$
and $\varepsilon>0$.  For every $z_0\in K$ there exists 
$j_\varepsilon$ such that $u(z_0)-\varepsilon <q_{j_\varepsilon}(z_0)$.
Since $u\in \USC(\C^n)$ and $q_{j_\varepsilon}\in \LSC(K_j)$ it follows
that we have   $u(z)-\varepsilon <q_{j_\varepsilon}(z)$ for all $z\in
K_j\cap U_0$ for  some  neighborhood $U_0$ of $z_0$ in $\C^n$.  
A simple compactness argument gives that 
there exists an open neighborhood $U$ of $K$ such that  
$u-\varepsilon <q_j$ on $K_j\subset U$ for
$j\geq j_\varepsilon$, possibly with $j_\varepsilon$ replaced by a larger
number.  Hence, $u-\varepsilon \leq V$, and since $\varepsilon$ is
arbitrary   we conclude that  $u\leq V$.  By taking supremum over 
$u$ we get $V^S_{K,q}\leq V$.   The same argument for
$\log|p|^{1/m}$, $p\in {\mathcal P}^S_m(\C^n)$,  in the role of $u$
implies that  $\Phi^S_{K,q}\leq \lim_{j\to \infty}\Phi^S_{K_j,q_j}$
by Proposition \ref{prop:2.2}.
\end{proof}

\section{Regularity of Siciak-Zakharyuta functions}
\label{sec:05}

In this section we study regularity of 
the functions $\Phi^S_{E,q}$ and $V^S_{E,q}$, where
we assume that $S$ is a convex subset of $\R^n_+$ with $0\in S$
and that $q\colon E\to  \R\cup\{+\infty\}$ is a
function on a subset $E$  of $\C^n$.  
Since the  Siciak function $\Phi^S_{E,q}$ is the supremum 
of a subclass of ${\mathcal C}(\C^n)$ we have
$\Phi^S_{E,q}\in\LSC(\C^n)$.  
If $q$ is an admissible weight on a closed set $E$,
then by Proposition \ref{prop:4.6} we have 
$V^{S*}_{E,q}\in \L^S_+(\C^n)$.

If  $V^{S*}_{E,q}\leq q$ on $E$ then 
$V^S_{E,q}=V^{S*}_{E,q}$ and we conclude that
  $V^S_{E,q}$ is continuous at
every point where it is lower semicontinuous. 
In this section we  prove that 
$V^S_{E,q}\in \LSC(\C^{*n})$ for every admissible weight on 
a compact set $K$. 
For that purpose we need to discuss
regularization of plurisubharmonic functions, but we begin with
local $\L$-regularity of sets.

\begin{definition}\label{def:5.1}  
A subset $E$  of $\C^n$ is said to be {\it $\L$-regular} at a point
$a\in \overline E$ if $V_E$ is continuous at $a$ and $E$ is said to be
{\it locally $\L$-regular} at $a$ if $E\cap U$ is
$\L$-regular at $a$ for every open neighborhood $U$ of $a$.   We say that 
$E$ is ({\it locally}) {\it $\L$-regular} if $E$  is 
(locally) $\L$-regular at every  $a\in \overline E$.
\end{definition}

\medskip
The function $V_E$ is continuous at  every interior point
of $E$. Moreover, $V_E$ is continuous at 
 $a\in \overline E$   if and only
if $V_E^*(a)=0$.   Thus, it is sufficient to check the condition for local
$\L$-regularity at boundary points only.

\begin{lemma} \label{lem:5.2}
Let $E$ be a closed  subset of  $\C^n$, 
$a\in E$ and assume that 
there exists a norm  with closed unit ball $\B$ 
such that for some $\delta>0$ and $b\in E$ we have
$a\in B=b+\delta\B\subset E$.
Then  $E$ is locally $\L$-regular at $a$.
\end{lemma}

\begin{proof}  Let $\|\cdot\|$ denote the norm and
$U$ be open with $a\in U$.
Choose $\tau>0$ so small that $C=c+\tau\delta \B\subset U$, where
$c=(1-\tau) a+\tau b$.  
We have $C\subset B\subset E$, so
$$0\leq V^*_E(a)\leq V_C^*(a)=\log^+(\|a-c\|/\tau\delta)
=\log^+(\|a-b\|/\delta)=0.
$$
\end{proof}

Observe that the lemma implies that every set of the form
$E=A+\delta \B$ is locally $\L$-regular, where $A\subset \C^n$ 
is closed and $\B$ is the closed unit ball with respect to some norm.
The following result is a generalization of 
Siciak \cite[Proposition 2.16]{Sic:1981}.

\begin{proposition} \label{prop:5.3}  
For every continuous function  $q$ on a locally $\L$-regular
closed subset $E$ of $\C^n$ we have $V^{S*}_{E,q}\leq q$
and consequently $V^{S}_{E,q}=V^{S*}_{E,q}$.
\end{proposition}

\begin{proof} Let $a\in E$ and take $\varepsilon >0$. 
Since $q$ is continuous on $E$ 
there exists an  open neighborhood $U$ of $a$
such that $q(z)\leq q(a)+\varepsilon$ for 
every $z\in E\cap U$.
Since $E$ is locally $\L$-regular we have
$V_{E \cap U}^*(a)=0$ and 
$$
V^{S*}_{E,q}(a)\leq V^{S*}_{E\cap U,q(a)+\varepsilon}(a)
\leq \sigma_S V^*_{E\cap U}(a)+q(a)+\varepsilon
=q(a)+\varepsilon.
$$
Since $a$ and $\varepsilon$ are arbitrary the inequality holds.  
\end{proof}

Convolution is a standard tool for  approximating functions
$u\in \PSH(\C^n)$.  We define a convolution operator 
$\Loneloc(\C^n)\to \Loneloc(\C^n)$, $u\mapsto u*\mu$, 
for a given positive Borel measure $\mu$ with compact support, 
\begin{equation} \label{eq:5.1}
   u*\mu(z)=\int_{\C^n} u(z-w)\, d\mu(w),
  \qquad z\in \C^n.
\end{equation}
If $u\in \PSH(\C^n)$ then $u*\mu \in \PSH(\C^n)$ and if $\mu$ is a 
probability measure then $u\mapsto u*\mu$  preserves $\mathcal L(\C^n)$.
In the case that $\mu$ is presented by a $\mathcal C^\infty$ function,
$\mu = \psi\, d\lambda$, then $u*\mu = u*\psi$ is a 
$\mathcal C^\infty$ function.
If we define 
$\psi_\delta\in {\mathcal C}_0^\infty(\C^n)$  by  
$\psi_\delta(z)=\delta^{-2n}\psi(z/\delta)$, $\delta>0$, 
with $\psi \in {\mathcal C}_0^\infty(\C^n)$
radially symmetric, $\psi\geq 0$,  and $\int_{\C^n}\psi\, d\lambda
=1$,    
then 
$u*\psi_\delta \searrow u$ as $\delta\searrow 0$. 
See Klimek \cite{Kli:1991} or
Hörmander \cite{Hormander:LPDO,Hormander:SCV}.
Siciak
\cite[Proposition 2.12]{Sic:1981} used convolution 
to prove that $V_{E,q}\in \LSC(\C^n)$ for compact $E$ and $q\in \LSC(E)$.
In general, the class $\L^S(\C^n)$ is not preserved under convolution (cf.~Theorem \ref{thm:5.8}).   

In order to preserve a  particular subclass of 
$\PSH(\C^n)$ under reg\-u\-lar\-iza\-tion,  
special  methods are sometimes needed. 
For example homogeneity is preserved under
\begin{equation} \label{eq:5.2}
 R_\delta u(z)=\int_G u(Az)\psi_\delta(A) \, d\mu(A), \qquad z\in \C^n,
\end{equation}
where $G$ is some group of $n\times n$ matrices with real or complex
entries and $\mu$ is a positive measure on the matrix space 
$\R^{n\times n}$ or $\C^{n\times n}$.     The smoothing kernel
$\psi_\delta$ is chosen so that it converges 
to the Dirac measure $\delta_I$ at the identity $I$
as $\delta\to 0$. 
This method  only gives a $\mathcal C^\infty$ function on 
$\C^n\setminus\{0\}$ when we integrate over the group of
{\it complex} invertible matrices, and it gives a  $\mathcal C^\infty$ function on 
$\C^n \setminus \C\R^n$, where
$\C\R^n = \{\lambda x; \lambda \in \C, x\in \R^n\}$, when we integrate over the group of {\it real}
invertible matrices.
See Sigurðsson \cite{Sig:1986} and  
Hörmander and Sigurðsson \cite{HorSig:1998}.

In order to preserve the classes $\L^S(\C^n)$ for a compact convex 
$S\subset \R^n_+$ with $0\in S$ it is natural to choose $G$ as 
the group of invertible diagonal matrices.  This group can be identified with
$\C^{*n}$ with coordinate wise multi\-pli\-cation, so it is natural to
choose $\mu$ as the Lebesgue measure $\lambda$ on $\C^n$.  

In the following text we allow us a slight abuse of notation 
by identifying a vector denoted by a lower case letter with
a diagonal matrix denoted by the corresponding upper case letter.
Thus, we identify the vector $a\in \C^n$ with the diagonal matrix
$A$ with diagonal $a$ and  in particular the vector 
$\mathbf 1$ with the identity matrix $I$. We define
\begin{align}\label{eq:5.3}
  R_\delta u(z)&=\int_{\C^n} u(Az) \psi_\delta(A)\, d\lambda(A)
=\int_{\C^n} u((I+\delta B)z) \psi(B)\, d\lambda(B)
\\
&=\int_{\C^n} u((1+\delta w_1)z_1,\dots,(1+\delta w_n)z_n) 
\psi(w)\, d\lambda(w), \qquad   z\in \C^n,
\nonumber
\end{align}
where we choose a function $0\leq \psi\in \mathcal C^\infty_0(\C^n)$, rotationally
symmetric in each variable with  $\int_{\C^n}\psi\, d\lambda=1$,
and set $\psi_\delta(z)=\delta^{-2n}\psi((z-{\mathbf 1})/\delta)$. 
By the Fubini-Tonelli theorem 
$R_\delta\colon \Loneloc(\C^n)\to \Loneloc(\C^n)$ and  $R_\delta u\to u$ in 
the $\Loneloc$ topology as $\delta\to 0$ and with local uniform
convergence for $u\in {\mathcal C}(\C^n)$.  Furthermore, 
$R_\delta\colon \PSH(\C^n)\to \PSH(\C^n)$.
This  implies that
if $u\in \L^S(\C^n)$  then
$R_\delta u\in \L^S(\C^n)$,     
more precisely, if $u\leq c_u+H_S$,  Proposition \ref{prop:3.2} implies
\begin{align*}
  R_\delta u(z) &\leq \int_{\C^n}
\big(c_u+H_S({\mathbf 1}+\delta w)+H_S(z)\big) \psi(w) \, d\lambda(w)\\
&\leq c_{u}+C\sigma_S\delta+H_S(z), \qquad z\in \C^n, 
\end{align*}
where $C=\sup_{w\in \supp \psi}\|w\|_\infty$.
The linear map $\C^n\to \C^n$, $A\mapsto Az$, has the 
Jacobi determinant $|z_1\cdots z_n|^2$, so for every 
$z\in \C^{*n}$ and corresponding matrix $Z$ with $z$ on the diagonal
we have
\begin{align*} 
  R_\delta u(z)&= \int_{\C ^n} u(w) \psi_\delta(Z^{-1}w) \, 
|z_1\cdots z_n|^{-2} \, d\lambda(w)
\\
&=\int_{\C ^n} u(w) \psi_\delta(w_1/z_1,\dots,w_n/z_n) \, 
|z_1\cdots z_n|^{-2} \, d\lambda(w).
\end{align*}
By applying the Lebesgue dominated convergence theorem 
we may differ\-en\-ti\-ate with respect to $z_j$ under the integral
sign infinitely often, so
this shows that for every $u\in \Loneloc(\C^n)$ we
have $R_\delta u\in \mathcal C^\infty(\C^{*n})$.  

\begin{proposition} \label{prop:5.4}  
Let $S$ be a compact convex subset of $\R^n_+$ with $0\in S$ and 
$q$ be an admissible weight on a compact subset 
$K$ of $\C^n$.  Then 
\begin{enumerate}
\item [{\bf (i)}]  $V^{S}_{K,q}\in \LSC(\C^{*n})$, and 
  if $\L^S(\C^n)$ is preserved under convolution (see Theorem \ref{thm:5.8}) then
 $V^{S}_{K,q}\in \LSC(\C^{n})$. 
\item[{\bf (ii)}]  If $V^{S*}_{K,q}\leq q$  on $K$ then
$V^{S}_{K,q}\in \L^S(\C^n)\cap {\mathcal C}(\C^{*n})$ 
and if $\log\Phi^S_{K,q}=V^S_{K,q}$ then $V^S_{K,q}\in {\mathcal C}(\C^n)$.
\end{enumerate}
\end{proposition}

\smallskip
\begin{proof}
\smallskip \noindent
{\bf (i)} \ 
It is sufficient to show that there exists an increasing
sequence $(u_j)_{j\in \N}$ in $\L^S(\C^n)\cap {\mathcal C}(\C^{*n})$
such that $u_j\leq q$ on $K$ for every $j$ and such that $\lim_{j\to \infty}
u_j(z)=V^S_{K,q}(z)$ for every $z\in \C^{*n}$.  (See Klimek
\cite[Section 2.3]{Kli:1991}.)  We take $u\in \L^S(\C^n)$, $\varepsilon>0$,
and let $R_\delta$ be the regularization operator \eqref{eq:5.3}.
Since $R_\delta u \searrow u$ and $q\in \LSC(K)$  
there exists $\delta_\varepsilon$  such
that  $u-\varepsilon \leq R_\delta u-\varepsilon <q$ on $K$ for every
$\delta \leq \delta_\varepsilon$.  These estimates tell us that there
exists a family $\F\subseteq \L^S(\C^n) \cap {\mathcal
  C}^\infty(\C^{*n})$ such that $V^S_{K,q}=\sup \F$.  By the Choquet
lemma  $V^S_{K,q}=\sup \G$ for some countable subfamily $\mathcal G$.
By arranging the elements of $\mathcal G$ into a sequence
$(v_j)_{j\in \N}$ and then setting $u_j=\max_{k\leq j}v_k$, we have
$u_j\nearrow V^S_{K,q}$ as $j\to \infty$. Hence, 
$V^S_{K,q}\in \LSC(\C^{*n})$.  This is a modification of 
the proof of \cite[Proposition 2.12]{Sic:1981}
where the regularization operator is given by \eqref{eq:5.1}, and that
is the second statement. 

\smallskip \noindent
{\bf (ii)} \ We have $V^{S*}_{K,q}\geq V^S_{K,q}$ and by 
 the definition of $V^S_{K,q}$ we have 
$V^{S*}_{K,q}\leq V^S_{K,q}$.  Hence, $V^S_{K,q}=V^{S*}_{K,q}\in
\USC(\C^{n})$ and from {\bf (i)} it follows that $V^S_{K,q}\in
{\mathcal C}(\C^{*n})$. 
The last statement follows from the fact that 
$\Phi^S_{E,q}\in \LSC(\C^n)$.
\end{proof}

\medskip
Our next task is to characterize the $\L^S(\C^n)$ classes that are
invariant under the convolution operator $u\mapsto u*\psi$ given by \eqref{eq:5.1}.
For that purpose we need to review a few facts from
convexity theory and define the hull of a convex set in $\R^n_+$ with
respect to a cone.  Recall that a point
$x$ in a convex set $S$ is said to be an {\it extreme point of $S$}
if the only representation of $x$ as a convex combination 
$x=(1-t)a+tb$ with  $a,b\in S$ and $t\in ]0,1[$ is the case
$x=a=b$.  We let $\ext\,  S$ denote the set of all extreme  
points in the convex set $S$.
Note that by the Minkowski theorem
\cite[Theorem 2.1.9]{Hormander:convexity}
every non-empty compact convex set $S$ is the closed convex 
hull of its extreme points and 
\begin{equation}
  \label{eq:5.4}
  \varphi_S(\xi)=\max_{x\in\ext\, S} \scalar x\xi, \qquad  \xi\in \R^n.
\end{equation}
Every affine hyperplane
$\{x\in \R^n\,;\, \scalar x\xi=\varphi_S(\xi)\}$   with $\xi\in \R^{n*}$
is called a {\it supporting hyperplane
  of the set $S$}.
For every  $s\in \partial S$ the set
$N^S_s=\{\xi\in \R^n\,;\, \scalar s\xi=\varphi_S(\xi)\}$ is a convex cone
which is called the {\it normal cone of $S$ at the point $s$}.
For every $\xi$ the upper bound in (\ref{eq:5.4}) is attained at
some  $s\in \ext\, S$,  so  $\R^n=\bigcup_{s\in\ext\, S}N_s^S$.

\smallskip
Every normal cone $N^S_s$ is an intersection of closed half spaces in
$\R^n$ with the origin in the boundary hyperplanes, that is
$N^S_s=\bigcap_{t\in \partial S} \{\xi\in \R^n \,;\,
  \scalar{s-t}\xi\geq 0\}$.
Observe that $\partial S$ can be replaced by the set of extreme points
$\ext\, S$.  
In the case $S$ is a convex polytope and $s$ is
an extreme point, then $N_s^S$ is an intersection
of finitely many half spaces in $\R^n$ with boundary 
hyperplanes containing the origin.  
See Figure 2(b).

\smallskip
If $\xi\in \R^{n*}$ and $\xi_0\in\R$, then 
$|\scalar  x\xi+\xi_0|/|\xi|$ is the distance from 
the point $x\in \R^n$ to the hyperplane
$\{y\in \R^n\,;\, \scalar y\xi+\xi_0=0\}$.
For supporting hyperplanes with normal $\xi$ we have
$\xi_0=-\varphi_S(\xi)$.   Hence,
$|\varphi_S(\xi)|/|\xi|$ is
the euclidean distance from the origin in $\R^n$ to the supporting hyperplane
$\{x\in \R^n \,;\, \scalar x\xi=\varphi_S(\xi)\}$.

\smallskip
Recall that a subset $\Gamma$ of $\R^n$ is said to be a cone
if $t\xi\in \Gamma$ for every $\xi\in \Gamma$ and $t\in \R_+$. 
The  {\it dual cone $\Gamma^\circ=\{x\in \R^n\,;\, \scalar x\xi\geq 0
\  \forall
  \xi\in \Gamma \}$} of $\Gamma$ is a closed convex cone
and if  $\Gamma\neq \R^n$ is closed and convex,
then $\Gamma^{\circ\circ}=\Gamma$.

\begin{definition} \label{def:5.5}
For every cone $\Gamma\subset (\R^n\setminus \R^n_-)\cup\{0\}$  
with $\Gamma\neq \{0\}$  
and  every subset $S$ of $\R^n_+$ we define the 
{\it $\Gamma$-hull} of $S$ by  
\begin{equation*}
\widehat S_\Gamma=\{x\in \R^n_+ \,;\, \scalar x\xi\leq \varphi_S(\xi)
 \ \forall \xi\in \Gamma\},
\end{equation*}
 and we say that $S$ is {\it $\Gamma$-convex} if $S=\widehat S_\Gamma$.
\end{definition}

\smallskip
Note that if $\Gamma_1\subseteq \Gamma_2$
then  $\widehat S_{\Gamma_2}\subseteq \widehat S_{\Gamma_1}$.
For every  compact and convex $S$ we have 
$S = \{x\in \R^n \,;\, \scalar x\xi\leq \varphi_S(\xi)
 \ \forall \xi\in \R^n\}=\widehat S_{\R^n}$, which implies that 
$S\subseteq \widehat S_\Gamma$ for every  cone $\Gamma\subseteq \R^n$.

\begin{proposition} \label{prop:5.6}
Let $S$ be a compact convex subset of $\R^n_+$ with $0\in S$ and
$\Gamma$ be a proper closed convex  cone containing at least one point in 
$\R^{*n}_+$.  Then 
\begin{equation*}
\widehat S_\Gamma=(S-\Gamma^\circ)\cap \R^n_+,
\end{equation*} 
and   $\varphi_{\widehat S_\Gamma}(\xi)=\varphi_S(\xi)$ holds
for every $\xi\in \Gamma$. If for every $\xi\in \R^n$ 
and every extreme point $x$ of $\widehat S_\Gamma$ 
there exists  $\eta\in \Gamma$ such that 
$\scalar x\xi\leq \scalar x\eta$ and $\varphi_S(\eta)=\varphi_S(\xi)$,  
then $S$ is $\Gamma$-convex. 
\end{proposition}

\begin{proof}  Take $a=s-t\in S'=(S-\Gamma^\circ)\cap \R^n_+$ with
$s\in S$ and $t\in \Gamma^\circ$.   For every $\xi\in \Gamma$
we have $\scalar t\xi\geq 0$ which implies 
$\scalar a\xi=\scalar s\xi-\scalar t\xi\leq \varphi_S(\xi)$
and $a\in \widehat S_\Gamma$.  

Conversely, we take  $a\not\in S'$ and prove
that $a\not \in \widehat S_\Gamma$.    Without restriction, we may assume
that $a\in \R^n_+$.
Since $S-\Gamma^\circ$ is convex
the Hahn-Banach theorem implies that
$\{a\}$ and $S-\Gamma^\circ$ can be separated
by an affine hyperplane.   Hence, 
there exists $\xi\in \R^n\setminus \{0\}$ and $c\in \R$ such
that $\scalar a\xi>c$ and $\scalar x\xi\leq c$ for every $x\in S-\Gamma^\circ$.
By replacing $c$ by $\sup_{x\in S-\Gamma^\circ}\scalar x\xi$
we may assume that 
there exists $s\in S$ and $t\in \Gamma^\circ$ such that
$\scalar {s-t}\xi=c$.   Now we need to prove that
$\xi\in \Gamma=\Gamma^{\circ\circ}$ by showing that  $\scalar y\xi\geq
0$ for every $y\in \Gamma^\circ$.    Since $\Gamma^\circ$ is a convex
cone, we have $t+y\in \Gamma^\circ$ and 
$c-\scalar y\xi=\scalar{s-t-y}\xi \leq c$. Hence,
$\scalar y\xi\geq 0$.
This implies that  $\scalar a\xi>c\geq \varphi_{S}(\xi)$ and we
conclude that  $a\not\in \widehat S_\Gamma$.  

Let $\xi \in \Gamma$ and take  $a=s-t\in \widehat S_\Gamma$, 
such that  $s\in S$, $t\in\Gamma^\circ$,
and $\varphi_{\widehat S_\Gamma}(\xi)=\scalar{s-t}\xi$.
Since $\scalar s\xi \leq 
\varphi_S(\xi) \leq \varphi_{\widehat S_\Gamma}(\xi)
=\scalar s\xi-\scalar t\xi$ and $\scalar t\xi\geq 0$,
we conclude that $t=0$ and 
$\varphi_S(\xi) = \varphi_{\widehat S_\Gamma}(\xi)$.
If for every $ \xi\in \R^n$ 
and every extreme point $x$ of $\widehat S_\Gamma$ 
there exists 
$\eta\in \Gamma$ such that 
$\scalar x\xi\leq \scalar x\eta$ and  
$\varphi_S(\eta)=\varphi_S(\xi)$, then
$\widehat S_{\Gamma}=\widehat S_{\R^n}=S$.
\end{proof}

\begin{definition}\label{def:5.7}
We say that a subset $S$ of $\R^n_+$ is a {\it lower set}
if for every $s\in S$ 
the cube $C_s=[0,s_1]\times\cdots\times[0,s_n]$
is a subset of $S$ and we call 
\begin{equation*}
  \widehat S_{\R^n_+}=\bigcup_{s\in S}C_s=\big(S-\R^n_+)\cap \R^n_+
\end{equation*}
the {\it lower hull of $S$}.   
\end{definition}

We have a characterization of lower sets:
\begin{theorem}\label{thm:5.8} 
Let $S\subset \R^n_+$ be compact and convex with $0\in S$
and set $\sigma_S=\varphi_S({\mathbf 1})$. Then
the following are equivalent:
\begin{enumerate}
\item[{\bf (i)}] $S\subset \R_+^n$ is a lower set,
\item[{\bf (ii)}] $S=\widehat  S_{\R^n_+}$,
\item[{\bf (iii)}] $\varphi_S(\xi)=\varphi_S(\xi^+)$ for every  $\xi\in \R^n$,
\item [{\bf(iv)}]
$H_S(z-w)\leq  \sigma_S\|w\|_\infty+H_S(z)$ 
for every $z,w\in \C^n$, 
\item [{\bf(v)}] $\L^S(\C^n)$ is translation invariant,  and 
\item [{\bf(vi)}] if $u\in \L^S(\C^n)$ then $u*\psi\in \L^S(\C^n)$ 
for every $\psi\in {\mathcal C}_0^\infty(\C^n)$
with $\psi\geq 0$ and $\int_{\C^n} \psi\, d\lambda=1$.
\end{enumerate}
\end{theorem}

\begin{proof} 
{\bf (i)}$\Leftrightarrow${\bf (ii)}$\Leftrightarrow${\bf (iii):}
Observe that the  extreme points of  $C_s$ are 
$t=(t_1,\dots,t_n)$ with $t_j=0$ or $t_j=s_j$, so
$\varphi_{C_s}(\xi)=\sup_{t\in\operatorname{ext} C_s} \scalar t\xi
=\scalar s{\xi^+}$ for every $\xi \in \R^n$.
Hence, for  every lower set $S$
we have  $\varphi_{S}(\xi)=\sup_{s\in S}\scalar s{\xi^+}
=\varphi_{S}(\xi^+)$ for $\xi \in \R^n$.
This equivalence follows from Proposition \ref{prop:5.6} with $\eta=\xi^+$.

\smallskip\noindent
{\bf (iii)}$\Rightarrow${\bf (iv):}
Let $w\in \C^n$, $z\in \C^{*n}$ and assume that $z-w\in \C^{*n}$.
Since $\varphi_S(\xi)-\varphi_S(\eta)\leq 
\varphi_S(\xi-\eta)$ and  $\varphi(\xi)\leq \sigma_S\|\xi\|_\infty$
for every $\xi,\eta\in \R^n$
and 
$|\log^+x-\log^+y|\leq |x-y|$ for every $x,y\in \R_+$, {\bf (iii)}
implies that 
\begin{multline*}
H_S(z-w)-H_S(z)=\varphi_S(\Log^+(z-w))-\varphi_S(\Log^+(z))\\
\leq\sigma_S\max_j|\log^+|z_j-w_j|-\log^+|z_j|| 
\leq \sigma_S \max_j||z_j-w_j|-|z_j|| 
\leq \sigma_S\|w\|_\infty.
\end{multline*}
By continuity of the function $H_S$ the inequality follows.

\smallskip\noindent
{\bf (iv)}$\Rightarrow${\bf (v):}  If  $u\leq c_u+H_S$,
then 
$u(\cdot-w)\leq c_u+\sigma_S\|w\|_\infty+H_S$ for every  $w\in \C^n$.

\smallskip\noindent
{\bf (v)}$\Rightarrow${\bf (vi):}
It is sufficient to show that $H_S*\psi\in \L^S(\C^n)$.
We observe that for every $\gamma\in ]0,1[$ the Riemann sum
$A_\gamma=\sum_{\alpha\in \Z^{2n}}\psi(\gamma\alpha) \gamma^{-2n}$ tends
to  $1=\int_{\C^n} \psi\, d\lambda$ as $\gamma\to 0$.  This  implies
that  the function $u_\gamma\colon \C^n\to \R$ defined by the Riemann
sum 
$$
u_\gamma(z)=\sum_{\alpha\in \Z^{2n}}H_S(z-\gamma\alpha)\psi(\gamma\alpha)
\gamma^{-2n}/A_\gamma 
$$
tends to $H_S*\psi$ as $\gamma\to 0$.  
By {\bf (v)} the
function $u_\gamma$ is in $\L^S(\C^n)$ for it is a convex combination
of functions in $\L^S(\C^n)$ and we have
$u_\gamma-c_\gamma\leq V^S_{\overline \D^n}=H_S$, where
$c_\gamma=\sup_{\overline \D^n}u_\gamma$.    Furthermore, 
since $H_S\in {\mathcal C}(\C^n)$
the convergence is locally
uniform, so by Proposition \ref{prop:4.3}  we have 
$H_S*\psi(z)\leq \sup_{\overline \D^n}H_S*\psi+H_S(z)$
for $z\in \C^n$ and conclude that $H_S*\psi\in \L^S(\C^n)$.

\smallskip\noindent
{\bf (vi)}$\Rightarrow${\bf (iii):}   We begin by taking
$\psi(\zeta)=\chi(\zeta_1)\cdots\chi(\zeta_n)$, where 
$0\leq \chi\in {\mathcal C}_0^\infty(\C)$ is rotationally invariant, $\supp
\chi\subset \overline \D$, and $\int_\C\chi\,
d\lambda=1$ and observe that with this choice of $\psi$
we have 
$(f*\psi)(z)=\sum_{j=1}^n
(f_j*\chi)(z_j)$
for any locally integrable function of the form 
$f(z)=\sum_{j=1}^n f_j(z_j)$. 

Let $\eta \in \R^{*n}$ have at least one strictly negative 
coordinate, enumerate the coordinates so that $\eta_j<0$ 
for $j=1,\dots,\ell$ and $\eta_j>0$ for $j=\ell+1,\dots,n$, 
and take $s\in S$ such that
$\varphi_S(\eta)=\scalar s\eta$. 
Then for every  
$t>0$ we have 
\begin{align*}
  H_S*\psi(e^{t\eta})-H_S(e^{t\eta})
&=\int_{\C^n}\varphi_S(\Log(e^{t\eta}-\zeta))\, \psi(\zeta)
\, d\lambda(\zeta) -\scalar s{t\eta}\\
&\geq \int_{\C^n}\scalar s
{\Log(e^{t\eta}-\zeta)-t\eta}\, \psi(\zeta)
\, d\lambda(\zeta)\\
&=\sum_{j=1}^n s_j \int_\D 
\left(\log|e^{t \eta_j}- \zeta_j| - t \eta_j 
\right) \chi(\zeta_j)\, d\lambda(\zeta_j)\\
&=-t\scalar{s'}{\eta'}
+\sum_{j=1}^\ell s_j\int_{\D}\log|e^{t\eta_j}-\zeta_j|
\chi(\zeta_j)\, d\lambda(\zeta_j) \\
&+\sum_{j=\ell+1}^n s_j\int_{\D}\log|1-e^{-t\eta_j}\zeta_j|
\chi(\zeta_j)\, d\lambda(\zeta_j).
\end{align*}
We let $v=\log|\cdot|\in \SH(\C)\cap \H(\C^*)$.  Then for
$j=1,\dots,\ell$
$$
\int_\D \log|e^{t \eta_j}-\zeta_j|\, \chi(\zeta_j) 
\, d\lambda(\zeta_j)=(v*\chi)(e^{t \eta_j})\to
v*\chi(0), \qquad t\to +\infty.
$$
Since $v\in \H(\C^*)$, $\chi$ is rotationally invariant, and 
$\supp \chi \subset \overline \D$, the mean value property
gives for $j=\ell+1,\dots,n$
$$
\int_{\D}\log|1-e^{-t\eta_j}\zeta_j|\, 
\chi(\zeta_j)\, d\lambda(\zeta_j)
=\log 1=0.
$$
Hence,
\begin{equation}
  \label{eq:5.5}
  H_S*\psi(e^{t\eta})-H_S(e^{t\eta})
\geq   -t\scalar{s'}{\eta'}
+\sum_{j=1}^\ell s_j(v*\chi)(e^{t\eta_j}).
\end{equation}
Assume that {\bf (iii)} does not hold and take $\xi_0\in \R^n$ such 
that $\varphi_S(\xi_0)<\varphi_S(\xi_0^+)$.  Then at least one
coordinate of $\xi_0$ is strictly negative. 
By continuity we can choose $\xi_0\in\R^{*n}$, 
and we can renumber the coordinates so that 
$\xi_{0,j}< 0$ for $j=1,\dots,\ell$ and $\xi_{0,j}>0$ for
$j=\ell+1,\dots,n$. By continuity  there exists an open neighborhood $U$
of $\xi_0^+$ such that $\varphi_S(\xi_0)<\varphi_S(\eta)$
for every $\eta\in U$. We fix $\eta=(\eta',\xi_0'')\in U$ with
$\eta_j<0$ for $j=1,\dots,\ell$.
There exists a point $s=(s',s'')\in \partial S$, 
such that $\scalar s\eta= \scalar{s'}{\eta'}+\scalar{s''}{\eta''}
=\varphi_S(\eta)$.  
We have $\scalar{s'}{\eta'}\leq 0$.  Equality is excluded for it would imply
$s'=0'$ and
$\varphi_S(\eta)=
\scalar {s''}{\xi_0''}=\scalar{(0', s'')}{\xi_0}
=\scalar s{\xi_0}
\leq \varphi_S(\xi_0)$,
contradicting the choice of $\eta$. 
Hence, $\scalar{s'}{\eta'}< 0$ and the estimate 
\eqref{eq:5.5} implies that $H_S*\psi-H_S$ is 
unbounded, contradicting  {\bf (vi)}.
\end{proof}

\section{Monge-Amp\`ere masses}

\label{sec:06}

The equilibrium measure for a bounded 
non-pluripolar set $E\subset\C^n$  is the Monge-Amp\`ere operator of
$V_E$, defined as  $\mu_E=(dd^cV^*_E)^n$ where 
$d^c=i(\overline{\partial}-\partial)$ and 
$(dd^cV^*_E)^n= dd^cV^*_E\wedge \cdots \wedge dd^cV^*_E$ is defined in
terms of currents. Similarly denote the  Monge-Amp\`ere measure 
of $V^{S*}_{E,q}$ by 
\begin{equation*}
  \mu^S_{E,q}= \big(dd^cV^{S*}_{E,q}\big)^n.
\end{equation*}

\begin{theorem}  \label{thm:6.1}
Let $S\subset \R^n_+$ be compact and convex with 
$0\in S$, $E \subset \C^n$ be closed, 
and $q$ be an admissible weight on $E$.  Then
  \begin{enumerate}
  \item [{\bf (i)}]
$\supp \mu^S_{E,q}\subseteq \{z\in E\,;\, V^{S*}_{E,q}(z) \geq
q(z)\}$, and  
  \item [{\bf (ii)}] $\{z\in E \,;\, V^{S*}_{E,q}(z) > q(z)\}$ 
is pluripolar.
\end{enumerate}
\end{theorem}

\smallskip\noindent
\begin{proof} {\bf (i)} \   We need to prove that 
$V^{S*}_{E,q}$ is maximal in 
$U=(\C^n\setminus E)\cup \{z\in E \,;\, V^{S*}_{E,q}(z)<q(z)\}$, which
is open.   Take $a\in U$.  If $a\in\C^n\setminus E$ then we 
take $r>0$ such that $B(a,r)\in \C^n\setminus E$.
If on the other hand $a\in E$ and $V^{S*}_{E,q}(a)<q(a)$
then by upper semicontinuity of $V^{S*}_{E,q}$ and
lower semicontinuity of $q$ there 
exists $r>0$ such  that
\begin{equation*}
  \sup_{\zeta\in B(a,r) }V^{S*}_{E,q}(\zeta)<
\inf_{\zeta\in E\cap B(a,r) } q(\zeta).
\end{equation*}
We need to prove that the restriction of $\mu^{S}_{E,q}$  
to some  $B(a,r)$ is the zero measure.
We let $V\in \PSH(\C^n)$ be the function given by
\begin{equation*}
 V(z)=
  \begin{cases}
    V^{S*}_{E,q}(z), &z \in \C^n\setminus B(a,r),\\
u(z), & z\in B(a,r),
  \end{cases}
\end{equation*}
where $u$ is the Perron-Bremermann function on $B(a,r)$ with boundary
values $V^{S*}_{E,q}$, i.e.,
\begin{equation*}
  u(z)=\sup\{v(z) \,;\, v\in \PSH(B(a,r)),\,  v^*\leq V^{S*}_{E,q} 
\text{ on } \partial B(a,r)\}.
\end{equation*}
Then $V^{S*}_{E,q}\leq V$ and since $V$ only deviates 
from $V^{S*}_{E,q}$ on a compact set  we have $V\in \L^S(\C^n)$
by Proposition \ref{prop:4.6}.
Furthermore, in the case $V^{S*}_{E,q}(a)<q(a)$ 
the maximum principle  implies 
\begin{equation*}
  V(z) \leq   \sup_{\zeta \in B(a,r) }V^{S*}_{E,q}(\zeta)<
\inf_{\zeta\in B(a,r) } q(\zeta)\leq q(z), \qquad z\in B(a,r),
\end{equation*}
and it implies  that $V\leq q$ on $E$.
Hence, $V=V^{S*}_{E,q}$ and by Klimek
\cite[Theorem 4.4.1]{Kli:1991},
$\mu^{S}_{E,q}=\big(dd^cV^{S*}_{E,q}\big)^n=0$ 
on $B(a,r)$. 

\smallskip\noindent
{\bf (ii)} 
  Since $V^S_{E,q}\leq q$ on $E$ we have
$\{z\in E \,;\, q(z)<+\infty\} \subseteq \{z\in \C^n \,;\, 
 V^S_{E,q}<+\infty\}$
Since $q$ is admissible, the left-hand side is non-pluripolar, 
and then so is the right-hand side.  Since
$\L^S(\C^n)/\varphi_S({\mathbf 1}) \subset \L(\C^n)$ it follows from
 Klimek \cite[Proposition 5.2.1 and Theorems 5.2.4 and 4.7.6]{Kli:1991}
 the set 
$\{z\in E\,;\,  q < V^{S*}_{E,q}(z)\}\subseteq 
\{z\in \C^n\,;\,  V^S_{E,q}(z)< V^{S*}_{E,q}(z)\}$
    is pluripolar.
\end{proof}

\medskip
The complex Monge-Amp\`ere mass of 
$H_S$ can be described in terms of the real Monge-Amp\`ere mass of
$\varphi_S$. Let
$U_\ell=H_S*\psi_{\delta_\ell}\searrow H_S$, where 
$0<\delta_\ell\searrow 0$ and
$\psi_{\delta_\ell}(\zeta)=\delta_{\ell}^{-2n}\psi(\zeta/\delta_\ell)$,
$\psi(\zeta)=\chi(\zeta_1)\cdots\chi(\zeta_n)$, where 
$0\leq \chi\in {\mathcal C}_0^\infty(\C)$ is rotationally invariant, $\supp
\chi\subset \overline \D$, and $\int_\C\chi\,
d\lambda=1$.  Then  
$U_\ell\in {\mathcal C}^\infty(\C^n)\cap \PSH(\C^n)$ and
$U_\ell(z)=U_\ell(|z_1|,\dots, |z_n|)$ holds.  We set
$v_\ell(\xi)=U_\ell(e^{\xi_1}, \dots,
e^{\xi_n})$. 
Since $$\frac{\partial^2 U_\ell(z)}{\partial z_j \partial \overline{z}_k}=\frac{1}{4z_j\overline{z}_k}\cdot\frac{\partial^2 v_\ell(\xi)}{\partial \xi_j \partial \xi_k}, \qquad z\in \C^{*n},$$
it follows that
$$\det\left(\frac{\partial^2 U_\ell(z)}{\partial z_j \partial \overline{z}_k}\right)= \frac{1}{4^n |z_1\cdots z_n|^2}\det \left(\frac{\partial^2 v_\ell(\xi)}{\partial \xi_j \partial \xi_k}\right)\bigg|_{\xi=\Log\,z}.$$
With $z_j=e^{\xi_j+i\theta_j}$, $\theta_j\in [0,2\pi]$ we have
$$
\left(\tfrac{i}{2}\right)^n dz_1\wedge d\overline{z}_1 \wedge \cdots \wedge d z_n\wedge d\overline{z}_n = |z_1\cdots z_n|^2 d\xi d\theta$$
and on $\C^{*n}$ we have
$$(dd^c U_\ell)^n=
n!\det\left(\frac{\partial^2 v_\ell(\xi)}{\partial \xi_j \partial \xi_k} \right) d\xi d\theta.
$$
The real Monge-Amp\`ere measure of $v$, denoted $\mathcal{MA}_\R(v)$, 
is defined for $\mathcal C^2$ function by 
$$\mathcal{MA}_\R(v)=\det\left(\frac{\partial^2 v(\xi)}{\partial \xi_j \partial \xi_k} \right) d\xi,$$
and extended to convex functions by locally uniform limits. This is done
in Figalli
\cite[Proposition 2.6]{Fig:2017}.

Letting $\ell\to \infty$, we have for every Borel set $E\subset \R^n$
\begin{multline*}
\int\limits_{\Log^{-1}(E)}(dd^c H_S)^n\\
= n!\,\big(\mathcal{MA}_\R(\varphi_S)\otimes d\theta\big) 
(E\times[0,2\pi]^n)
=(2\pi)^nn!\,\mathcal{MA}_\R(\varphi_S) (E).
\end{multline*}
In particular 
\begin{equation}
\label{eq:6.1}
    \int_{\T^n}(dd^c H_S)^n= (2\pi)^n n!\,\mathcal{MA}_\R(\varphi_S)(\{0\}).
\end{equation}
This will be useful for our next result. Its proof is borrowed from
Rashkovskii \cite[Theorem 3.4]{Ras:2001}.  

\begin{theorem}\label{thm:6.2}
Let $S\subset \R^n_+$ be compact and convex with  $0\in S$, $E \subset
\C^n$ be closed, and $q$ be an admissible weight on $E$.
Then $$\mu^S_{E,q}(\C^n)=(2\pi)^n n! \operatorname{vol}(S),$$ where
$\operatorname{vol}$ denotes the euclidean volume. 
\end{theorem} 
\begin{proof} 
By Proposition \ref{prop:4.5}, $V^S_{K,q}-H_S$ is bounded. The
comparison 
principle, Klimek \cite[Theorem 3.7.1]{Kli:1991},
then implies that 
$$
\mu^S_{K,q}(\C^n) = \int_{\C^n}(dd^c H_S)^n= \mu^S_{\T^n,0}(\C^n).
$$ 
By Theorem \ref{thm:6.1} {\bf (i)},  
$\mu^S_{\T^n,0}(\C^n)=\int_{\T^n}(dd^c H_S)^n $. We already
established in 
\eqref{eq:6.1} that
$$\int_{\T^n}(dd^c H_S)^n= (2\pi)^nn!\mathcal{MA}_\R(\varphi_S)(\{0\})$$
By Blocki \cite{Blo:2009}, see also 
Figalli \cite{Fig:2017}, we have 
$$\mathcal{MA}_\R(\varphi_S)(\{0\})
= \operatorname{vol}(\{s\in \R^n\,;\,
\scalar s\xi \leq 
\varphi_S(\xi),\;\forall\xi\in \R^n\})=\operatorname{vol} (S).$$
\end{proof}

\section{Characterization of polynomials
 by $L^2$-estimates}
\label{sec:07}

In this section we study  characterization of the polynomial spaces
${\mathcal P}^S_m(\C^n)$ with weighted $L^2$-norms of entire functions.
Recall that the Liouville type Theorem~\ref{thm:3.6} 
states that if $f\in \OO(\C^n)$ and 
$|f|e^{-mH_S-a\log^+\|\cdot\|_\infty}\in L^\infty(\C^n)$ for some 
$a\in [0,d_m[$, where $d_m$ is the distance between 
$mS$ and $\N^n\setminus mS$ in the $L^1$-norm, 
then $f\in {\mathcal P}^S_m(\C^n)$.
We take a measurable function $\psi\colon \C^n\to
\overline \R$ and let  $L^2(\C^n,\psi)$ 
denote the space of all measurable $f\colon \C^n\to \C$ such that
\begin{equation}
  \label{eq:7.1}
   \|f\|_\psi^2=\int_{\C^n}|f|^2e^{-\psi}\, d\lambda <+\infty.
\end{equation}

\begin{proposition} \label{prop:7.1}
Let $f\in L^2(\C^n,\psi)\cap\OO(\C^n)$ for some
measurable $\psi\colon \C^n\to \overline \R$
with power series  expansion
$f(z)=\sum_{\alpha\in \N^n}a_\alpha z^\alpha$ at the origin.
Then for every polyannulus 
$A_{\sigma,\tau}=\{\zeta\in \C^n\,;\, e^{\sigma_j}\leq 
|\zeta_j|<e^{\tau_j}\}$ in $\C^n$, where
$\sigma,\tau \in \R^n$,  $\sigma_j<\tau_j$ for 
$j=1,\dots,n$,  with volume 
$v(A_{\sigma,\tau})=\pi^n\prod_{j=1}^n(e^{2\tau_j}-e^{2\sigma_j})$,   we have
\begin{equation}   \label{eq:7.2}
  |a_\alpha|
\leq \dfrac {\|f\|_\psi}{v(A_{\sigma,\tau})}
\bigg(\int_{A_{\sigma,\tau}} 
\dfrac{e^{\psi(\zeta)}}{|\zeta^\alpha|^2}
 \, d\lambda(\zeta)\bigg)^{1/2}.
\end{equation}
Furthermore, if  $\psi$ is rotationally invariant in each
variable $\zeta_j$, then for 
$K_{\sigma,\tau}=\prod_{j=1}^n [\sigma_j,\tau_j] \subset \R^n$
and
$\chi(\xi)=\tfrac 12 \psi(e^{\xi_1},\dots,e^{\xi_n})$
we have 
\begin{equation}  \label{eq:7.3}
  |a_\alpha| \leq 
  \dfrac {\|f\|_\psi}{\prod_{j=1}^n(1-e^{-2(\tau_j-\sigma_j)})}
e^{-\scalar{\mathbf 1}\tau}
\bigg(\int_{K_{\sigma,\tau}} 
e^{2(\chi(\xi)-\scalar \alpha\xi)} \, d\lambda(\xi)\bigg)^{1/2}.    
\end{equation}
\end{proposition}

\begin{proof} By the  Cauchy formula for derivatives 
\begin{multline*}
  a_\alpha 
=\dfrac 1{(2\pi i)^n}\int_{C_r} \dfrac{f(\zeta)}{\zeta^\alpha}\cdot
\dfrac{d\zeta_1\cdots d\zeta_n}{\zeta_1\cdots\zeta_n} \\
=\dfrac 1{(2\pi)^n}\int_{[-\pi,\pi]^n} 
\dfrac{f(r_1e^{i\theta_1},\dots,r_ne^{i\theta_n})}{r^\alpha
e^{i\scalar\alpha\theta}}\, d\theta_1\cdots d\theta_n,
\end{multline*}
where $C_r=\{z\in \C^n\,;\, |z_j|=r_j\}$, 
is any polycircle with  center $0$ and poly\-radius  $r\in \R_+^{*n}$.  
We parametrize $C_r$ by 
$[-\pi,\pi]^n\ni \theta\mapsto 
(r_1e^{i\theta_1},\dots,r_ne^{i\theta_n})$, 
 multiply the integral 
by $r_1\cdots r_n\, dr_1\cdots dr_n$, integrate with 
respect to  $r_j$ over 
$[e^{\sigma_j},e^{\tau_j}]$, note that
$\int_{[e^{\sigma_j},e^{\tau_j}]}r_j\, dr_j=\tfrac
12(e^{2\tau_j}-e^{2\sigma_j})$, set  
$L_{\sigma,\tau}=\prod_{j=1}^n\big([e^{\sigma_j},e^{\tau_j}]
\times[-\pi,\pi]\big)$, 
and  get
\begin{align*}
  a_\alpha 
&=\dfrac 1{v(A_{\sigma,\tau})}
\int_{L_{\sigma,\tau}} 
\dfrac{f(r_1e^{i\theta_1},\dots,r_ne^{i\theta_n})}{r^\alpha
e^{i\scalar\alpha\theta}}\, (r_1\, dr_1d\theta_1)\cdots (r_n\,
dr_nd\theta_n)
\\
&=\dfrac 1{v(A_{\sigma,\tau})} \int_{A_{\sigma,\tau}}
\dfrac{f(\zeta)}{\zeta^\alpha} \, d\lambda(\zeta)
=\dfrac 1{v(A_{\sigma,\tau})} \int_{A_{\sigma,\tau}}
f(\zeta)e^{-\psi(\zeta)/2} 
\cdot \dfrac{e^{\psi(\zeta)/2}}{\zeta^\alpha} \, d\lambda(\zeta).  
\end{align*}
Now (\ref{eq:7.2}) follows from the Cauchy-Schwarz inequality 
and (\ref{eq:7.3}) by inte\-grating over the angular variables
and also using the fact that   
$v(A_{\sigma,\tau})=\pi^n\prod_{j=1}^n(e^{2\tau_j}-e^{2\sigma_j})$.
\end{proof}

\begin{theorem} \label{thm:7.2}
Let $S$ be a compact convex subset of $\R_+^n$ with  $0\in S$,
$m\in\N^*$, and $d_m=d(mS,\N^n\setminus mS)$ denote
the euclidean distance between $mS$ and $\N^n\setminus mS$.
If $f\in \OO(\C^n)$ and for some 
$a\in[0,d_m[$
\begin{equation}
  \label{eq:7.4}
\int_{\C^n}|f|^2(1+|\zeta|^2)^{-a}e^{-2mH_S} \,
d\lambda <+\infty,
\end{equation}
then $f\in {\mathcal P}_m^{\widehat S_\Gamma}(\C^n)$, where
$\widehat S_\Gamma$ is the $\Gamma$-hull of $S$  and 
$\Gamma=\Gamma_a$ consists   of all $\xi$ such that 
the angle between the vectors ${\mathbf 1}=(1,\dots,1)$ and $\xi$ is 
$\leq \arccos(-(d_m-a)/\sqrt n)$.
\end{theorem}

	\begin{figure}[h]
		\centering
		\begin{tikzpicture}
      scale = 1.1, every node/.style={transform shape}]
			\clip (-2.8, -2.8) rectangle (2.8, 2.8);

			\draw[fill = gray!20]
				(-10, 3) --
				(0, 0) --
				(10, -35) --
				(10, 10);


			\draw[->] (0, 0) -- (2, 2);
			\node at (1.4, 1.8) {\small $\mathbf{1}$};
			\node at (2.3, -0.8) {\small $\Gamma$};

			\draw (0, -10) -- (0, 10);
			\draw (-10, 0) -- (10, 0);

			\node at (1.2, 0.5) {\small $\theta$};
			\path[clip] (1, 1) -- (0, 0) -- (10, -35) -- cycle; 
			\draw (0, 0) circle (1);
		\end{tikzpicture}

{\small Figure 1: The cone $\Gamma$ has  opening angle 
$\theta = \arccos(-(d_m - a)/\sqrt{n})$}
	\end{figure}

\noindent
Observe that the largest possible $d_m$ is $1$ and 
smallest possible $a$ is $0$, which implies that
the largest possible opening angle of the cone $\Gamma$
is $\arccos(-1/\sqrt n)$ in the case $d_m=1$ and 
$\Gamma=\R^n\setminus \R^{*n}_-$. 
If $a_0$ is the infimum of $a$ such that 
\eqref{eq:7.4} holds, then $\Gamma_{a_0} = 
\bigcup_{a>a_0} \Gamma_a$.
Therefore $f$ is a polynomial with exponents in 
$m \widehat S_{\Gamma(a_0)} = \bigcap_{a>a_0}
m\widehat S_{\Gamma(a)}$.

We are interested in conditions on
cones $\Lambda\subseteq \Gamma$
which guarantee that $\widehat S_\Lambda =S$:

\begin{corollary} \label{cor:7.7}  The function $f$ in
Theorem~\ref{thm:7.2} is in  ${\mathcal P}_m^{S}(\C^n)$
in the cases:
\begin{enumerate}
\item [{\bf(i)}] $S=\widehat S_{\Lambda}$ for some cone 
$\Lambda$ contained in 
$\{\xi\in \R^n \,;\, \scalar{\mathbf 1}\xi\geq 0\}$.
\item [{\bf(ii)}] $S$ is a lower set, that is $S=\widehat S_{\R^n_+}$.
\item [{\bf(iii)}]  $(mS)\cap \N^n=(m\widehat S_{\Gamma})\cap \N^n$. 
\end{enumerate}
\end{corollary}

\begin{prooftx}{Proof of Theorem \ref{thm:7.2}}
Let $f(z)=\sum_{\alpha\in \N^n} a_\alpha z^\alpha$ be
  the Taylor expansion of $f$ at the origin.  
We need to show
that $a_\alpha=0$ for every $\alpha\in \N^n\setminus m\widehat S_\Gamma$.
Since $\alpha\not\in m \widehat S_\Gamma$, there exists $\tau \in \Gamma$
such that $|\tau|=1$ and $\scalar \alpha\tau >m\varphi_S(\tau)$.
By rotating $\tau$ we may assume that $\tau$ is an interior point
of $\Gamma$.  Since the angle between $\tau$ and ${\mathbf 1}$ is
$\leq \arccos(-(d_m-a)/\sqrt n)$ we have 
$-\scalar{\mathbf 1}\tau <d_m-a$.
We choose $\varepsilon >0$ such that
 $d_m-a-\varepsilon>0$, and
$-\scalar{\mathbf 1}\tau <d_m-a-\varepsilon$.
Recall that $\scalar\alpha\tau-m\varphi_S(\tau)$ is the euclidean
distance from $\alpha$ to the supporting hyperplane
$\{x\,;\, \scalar x\tau=m\varphi_S(\tau)\}$, so by assumption
$m\varphi_S(\tau)-\scalar\alpha\tau\leq -d_m$.   Hence
\begin{equation*}
-\scalar{\mathbf 1}\tau + m\varphi_S(\tau)-\scalar\alpha\tau
<-a-\varepsilon.
\end{equation*}
We take $\sigma\in \R^n\setminus \{0\}$ with $\sigma_j<\tau_j$ for
$j=1,\dots,n$ such that
\begin{equation*}
-\scalar{\mathbf 1}\tau + m\varphi_S(\xi) -\scalar\alpha\xi
<-(a+\varepsilon)|\xi|, \qquad 
\xi\in K_{\sigma,\tau}=\prod_{j=1}^n[\sigma_j,\tau_j].
\end{equation*}
By homogeneity we get 
\begin{equation*}
-t\scalar{\mathbf 1}\tau + m\varphi_S(\xi)-\scalar\alpha\xi
<-(a+\varepsilon)|\xi|, \qquad t>0, \ \xi\in tK_{\sigma,\tau}.
\end{equation*}
Let $\xi_j=\log|\zeta_j|$ and observe that
$(1+|\zeta|^2)^a \leq (n+1)^a
\max\{1,\|\zeta\|_\infty^{2a} \}$. From this inequality and
\eqref{eq:7.4} it follows that
$f\in L^2(\C^n,\psi)$,  where
\begin{equation*}
\tfrac 12 \psi(\zeta)={a}\log\|\zeta\|_\infty+mH_S(\zeta)  
=a \|\xi\|_\infty+m\varphi_S(\xi).
\end{equation*}
We set $\chi(\xi)=\tfrac 12 \psi(e^{\xi_1},\dots,e^{\xi_n})$.
Then 
\begin{equation*}
-t\scalar{\mathbf 1}\tau + \chi(\xi)-\scalar\alpha\xi
<-\varepsilon|\xi|, \qquad t>0, \ \xi\in tK_{\sigma,\tau}.
\end{equation*}
The estimate (\ref{eq:7.3}) with $tK_{\sigma,\tau}$ in the
role  of $K_{\sigma,\tau}$ gives 
\begin{align*}
  |a_\alpha| &\leq 
  \dfrac {\|f\|_\psi}{\prod_{j=1}^n(1-e^{-2(\tau_j-\sigma_j)t})}
e^{-t\scalar{\mathbf 1}\tau}
\bigg(\int_{tK_{\sigma,\tau}} 
e^{2(\chi(\xi)-\scalar\alpha\xi)} \, d\lambda(\xi)\bigg)^{1/2}\\
&\leq \dfrac {\|f\|_\psi}{\prod_{j=1}^n(1-e^{-2(\tau_j-\sigma_j)t})}
e^{-\varepsilon|\sigma|t}t^{n/2} v(K_{\sigma,\tau})^{1/2} \to 0,
\qquad t\to +\infty,
\nonumber    
\end{align*}
and we conclude that $a_\alpha=0$.
\end{prooftx}

\begin{figure}[!h]
\def\mynda{0.1}
\def\myndb{0.93}
\def\myndm{5}
\def\myndscale{0.8} 
\centering
\begin{tikzpicture}[baseline={(2,2)},
scale = \myndscale, every node/.style={transform shape}]
	\clip (-1, -1) rectangle (7, 7);
	\draw[->] (-0.3, 0) -- (5.3, 0);
	\draw[->] (0, -0.3) -- (0, 5.3);

	\draw[fill = gray!20]
		(0, \myndm) --
		(\myndm*\myndb, \myndm - \myndm*\myndb) --
		(\myndm*\mynda, 0) --
		(0, 0) --
		(0, \myndm);
	\node at (1, 2) {$mS$};

	\draw[dotted] (\myndm*\myndb, \myndm - \myndm*\myndb) -- (\myndm, 0);

	\node[draw, fill, circle, inner sep = 0.5pt] at (0, \myndm) {};
	\node at (-0.5, \myndm) {\small $(0, m)$};

	\node[draw, fill, circle, inner sep = 0.5pt] at (\myndm*\mynda, 0) {};
	\node at (\myndm*\mynda +0.1, 0.3) {\small $(ma, 0)$};


	\node[draw, fill, circle, inner sep = 0.5pt] at (\myndm, 0) {};
	\node at (\myndm + 0.2, -0.2) {\small $(m, 0)$};

	\node[draw, fill, circle, inner sep = 0.5pt] at (0, 1) {};
	\node at (-0.5, 1) {\small $(0, 1)$};

	\node[draw, fill, circle, inner sep = 0.5pt] at (1, 0) {};
	\node at (1, -0.2) {\small $(1, 0)$};

	\node[draw, fill, circle, inner sep = 0.5pt] at (\myndm*\myndb, \myndm - \myndm*\myndb) {};
	\node at (\myndm*\myndb + 1, \myndm - \myndm*\myndb + 0.3) {\small $(mb, m(1\! -\! b))$};
\end{tikzpicture}
\begin{tikzpicture}[baseline={(0,0)},
scale = \myndscale, every node/.style={transform shape}]
	\clip (-2.8, -2.8) rectangle (2.8, 2.8);
	\draw[fill = gray!20]
		(0, 0) --
		(-10, 0) --
		(0, -10) --
		(0, 0);
	\node at (-1.4, -1.4) {\small $N^{S}_{(0,0)}$};

	\draw[fill = gray!20]
		(0, 0) --
		(-5, 0) --
		(-5, 5) --
		(5, 5) --
		(0, 0);
	\node at (-0.5, 1.4) {\small $N^{S}_{(0,1)}$};

	\draw[fill = gray!20]
		(0, 0) --
		(1 , {\myndm - \myndb*\myndm - 1*(\myndb - \mynda)/(1 - \myndb) - \myndm + \myndb*\myndm}) --
		(10, \myndm - \myndb*\myndm + 10 - \myndm + \myndb*\myndm) --
		(0, \myndm - \myndb*\myndm - \myndm + \myndb*\myndm);
	\node at (1.4, -0.2) {$N^{S}_{(b,(1-b))}$};

	\draw[fill = gray!20]
		(0, 0) --
		(1, {-1*(\myndb - \mynda)/(1 - \myndb)}) --
		(0, -4) --
		(0, 0);
	\node at (0.9, -2.5) {\small $N^{S}_{(a,0)}$};
\end{tikzpicture}

\small{Figure 2: (a) The set $S$. (b) The normal cones $N_s^S$ of the extreme 
points of $S$.}

\end{figure}

\bigskip\medskip
The $\Gamma$-hull 
$\widehat S_\Gamma$ can not be replaced by $S$ in 
Theorem \ref{thm:7.2}:
\begin{example}\label{ex:7.4} 
Let $m\geq 4$ and $S\subseteq \R_+^2$ be the quadrilateral 
\begin{equation*}
S=\ch\{(0,0), (a,0), (b,1-b),(0,1)\}.  
\end{equation*}
where $0<a<1/m$ , $0<a<b<1$, $m(1-b)<1$, and $(b-a)/(1-b)>m-2-am$.
Then $(1,0),(2,0), \dots , (m-3,0)\notin mS$, but the 
calculations below show that $\|p_k\|_{2mH_S}<+\infty$ for 
$p_k(z)=z^{(k,0)}=z_1^k$ with $k=1,\dots, m-3$.

Since the map $(\R\times ]-\pi,\pi[)^2\to 
\C^2$, $(\xi_1,\theta_1,\xi_2,\theta_2)\mapsto 
(e^{\xi_1+i\theta_1},e^{\xi_2+i\theta_2})$ has the Jacobi determinant
$e^{2\xi_1+2\xi_2}$, we have
\begin{gather*}
  \|p_k\|_{2mH_S}^2=\int_{\C^2}|z_1|^{2k}e^{-2mH_S(z)}\, d\lambda(z)
=4\pi^2 \int_{\R^2}e^{2(k+1)\xi_1+2\xi_2-2m\varphi_S(\xi)}\,
                 d\xi_1d\xi_2,
\\
 \varphi_S(\xi)=
	\max_{x\in \operatorname{ext}(S)} \langle x,\xi \rangle
	=
 \begin{cases}
 0, &\xi \in N^S_{(0,0)},\\   
a\xi_1, &\xi \in N^S_{(a,0)},\\
b\xi_1+(1-b)\xi_2,  &\xi \in N^S_{(b,(1-b))},\\
\xi_2, &\xi \in N^S_{(0,1)}.  
 \end{cases}  
\end{gather*}
We split the  integral over $\R^2$ into the sum of the
integrals over the normal cones at the extreme points of $S$, which we
calculate as
\begin{align*}
  \int_{N^S_{(0,0)}} e^{2(k+1)\xi_1+2\xi_2}\, d\xi
&=\int_{-\infty}^0 e^{2(k+1)\xi_1} \, d\xi_1 \,   \int_{-\infty}^0
e^{2\xi_2} \, d\xi_2=\dfrac 1{4(k+1)},\\
  \int_{N^S_{(a,0)}} e^{2(k+1)\xi_1+2\xi_2-2ma\xi_1}\, d\xi
&=\int_0^{\infty} e^{2(k+1-ma)\xi_1}\,
    \int_{-\infty}^{-\xi_1(b-a)/(1-b)} e^{2\xi_2}\, d\xi_2\, d\xi_1\\
& =\dfrac {1}{4((b-a)/(1-b)+ma-1-k)},
\end{align*}

\vspace{-0.3cm}

\begin{multline*}
 \int_{N^S_{(b,1-b)}} e^{2(k+1)\xi_1+2\xi_2-2m(b\xi_1+(1-b)\xi_2)}\,
   d\xi \\ 
=\int_0^{\infty} e^{2(k+1-mb)\xi_1}\,
\int_{-\xi_1(b-a)/(1-b)}^{\xi_1} e^{2(1-m(1-b))\xi_2}\, d\xi_2\,
                                  d\xi_1 \\
 =\dfrac {1}{4(1-m(1-b))}\bigg(
\dfrac 1{m-2-k}+\dfrac{1}{(b-a)/(1-b)+ma-1-k)}
\bigg),
\end{multline*}

\vspace{-0.2cm}

\begin{align*}
 \int_{N^S_{(0,1)}} e^{2(k+1)\xi_1+2\xi_2-2m\xi_2}\, d\xi\
&=\int_0^\infty e^{2(1-m)\xi_2}
\int_{-\infty}^{\xi_2}e^{2(k+1)\xi_1}\, d\xi_1 \, d\xi_2\\
&=\dfrac {1}{4(k+1)(m-2-k)}.
\end{align*}
This shows that \eqref{eq:7.4} is satisfied with 
$p_k$ in the role of $f$, but  $p_k\not\in \mathcal P^S_m(\C^n)$.
\end{example}

\section{Pullbacks of polynomial classes  by 
polynomial maps}
\label{sec:08}

\noindent
Let  $S\subset \R^n_+$ be  compact and convex with $0\in S$ and   
$f=(f_1,\dots,f_n)\colon \C^\ell\to \C^n$ be a polynomial map.
If $f_j(z)=\sum_{\alpha\in \N^\ell}a_{j,\alpha} z^\alpha$,
$I_j=\{\alpha \in \N^\ell\,;\, a_{j,\alpha}\neq 0\}$, and 
$S_j\subset \R^\ell_+$ be the convex hull of $I_j$.
Then $f_j\in {\mathcal P}^{S_j}_1(\C^\ell)$ and for 
every   $p(w)=\sum_{\beta\in (mS)\cap \N^n} b_\beta w^\beta$ in
${\mathcal P}^S_m(\C^n)$ we have 
  \begin{equation*}
(f^*p)(z)=\sum_{\beta\in (mS)\cap \N^n} b_\beta 
f_1(z)^{\beta_1}\cdots f_n(z)^{\beta_n}.
  \end{equation*}
By Theorem  \ref{thm:3.6}  we have  $|f_j(z)|\leq
e^{c_j+H_{S_j}(z)}$  which implies that for 
every $z\in \C^{*\ell}$ 
  \begin{align*}
|(f^*p)(z)|&\leq \sum_{\beta\in (mS)\cap \N^n} |b_\beta| 
e^{\scalar c\beta +\beta_1H_{S_1}(z)+\cdots+\beta_nH_{S_n}(z)}\\
&\leq \bigg(\sum_{\beta\in (mS)\cap \N^n} |b_\beta| 
e^{\scalar c\beta} \bigg)  
e^{m\varphi_S(\varphi_{S_1}(\Log\, z),\dots,\varphi_{S_n}(\Log\, z))}.
\nonumber
  \end{align*}
We have for every $\xi\in \R^\ell$ that 
\begin{multline*}
\varphi_S(\varphi_{S_1}(\xi),\dots, \varphi_{S_n}(\xi))\\
=
\sup\{  \langle x_1s_1+\cdots +
x_n s_n, \xi\rangle;\, x\in S,\, s_j\in S_j,\, j=1,\dots, n \}\\
= \sup_{x\in S} \varphi_{x_1S_1+\cdots+x_n S_n}(\xi)
=\varphi_{S'}(\xi),\nonumber
\end{multline*}
where $S'=\bigcup_{x\in S}x_1S_1+\cdots+x_n S_n$.
The set  $S'\subset \R^\ell_+$ 
is compact and convex.

\begin{proposition}\label{prop:8.1}
Assume that $\overline{S\cap \Q^n}=S$.  Then with the notation 
above $S'$ is the smallest 
compact convex subset $T$ of $\R^\ell_+$ with $0\in T$ for which
$f^*\big({\mathcal P}^S_m(\C^n)\big) \subseteq  {\mathcal P}^{T}_m(\C^\ell)$
for all $m\in \N$.
\end{proposition}

\begin{proof} Assume that $T \subsetneq S'$ such that 
$f^*\big({\mathcal P}^S_m(\C^n)\big)
\subseteq {\mathcal P}^T_m(\C^\ell)$  for every $m\in \N$.
Then there exists $\xi\in \R^\ell$ such that
$\varphi_T(\xi) < \varphi_{S'}(\xi)$ and  since  $S\cap \Q^n$ is dense 
in $S$ we can choose $r\in S\cap \Q^n$ such that 
\begin{equation}
  \label{eq:8.1}
  \varphi_T(\xi) < r_1\scalar{\alpha_1}\xi+\cdots
+r_n\scalar{\alpha_n}\xi \leq \varphi_{S'}(\xi).
\end{equation}
Now we fix  $m\in \N$ such that $\gamma = m r\in \N^n$,  
define  $p\in {\mathcal P}^S_\mu(\C^n)$ by $p(w)=w^\gamma$.
By Theorem \ref{thm:3.6}, there exists  a constant $c_p$ such that 
\begin{equation}
  \label{eq:8.2}
  |f^*p(z)|\leq e^{c_p+m H_T(z)}, \qquad z\in \C^\ell.
\end{equation}
Since $S_j$ is a convex polytope in $\R^\ell_+$, the set
$\bigcap_{j=1}^n \bigcup_{\alpha\in \ext S_j}\mathring{N}^{S_j}_\alpha$ is
dense in $\R^\ell$, where $\mathring{N}^{S_j}_{\alpha}$ is the interior
of $N^{S_j}_{\alpha}$. 
Hence $\xi$ may be chosen from $\mathring{N}^{S_j}_{\alpha_j}$
where  $\alpha_j\in \ext S_j$ and 
$\varphi_{S_j}(\xi)=\scalar{\alpha_j}\xi > 
\scalar {\alpha}\xi$ for  
every $\alpha\in I_j\setminus\{\alpha_j\}$.    
All the $\alpha_j$ are from $\N^\ell$ by the definition of $S_j$.
We define the sequence $(\zeta_{k})_{{k}\in \N}$ in $\C^\ell$ by
$\zeta_{k}=(e^{{k}\xi_1},\dots,e^{{k}\xi_\ell})$ and we will
show  that (\ref{eq:8.1}) implies that the sequence
$\big(f^*p(\zeta_{k})e^{-m H_T(\zeta_{k})}\big)_{{k}\in \N}$ 
is unbounded,   contradicting the estimate (\ref{eq:8.2}).
First we observe that
$\zeta_{k}^{\alpha_j}=e^{{k}\scalar{\alpha_j}\xi}$ and
then that (\ref{eq:8.1}) implies 
\begin{equation}
  \label{eq:8.3}
  (\zeta_{k}^{\alpha_1},\dots,\zeta_{k}^{\alpha_n})^\gamma
e^{-mH_T(\zeta_{k})}
=e^{{k}m(r_1\scalar{\alpha_1}\xi+\cdots+r_n\scalar{\alpha_n}\xi
-\varphi_T(\xi))} \to +\infty
\end{equation}
when ${k}\to +\infty$.
Next we observe that 
\begin{equation}
  \label{eq:8.4}
  f_j(\zeta_{k})/\zeta_{k}^{\alpha_j}=
a_{j,\alpha_j}+\sum_{\alpha\in I_j\setminus\{\alpha_j\}} a_{j,\alpha}
e^{-{k}(\scalar{\alpha_j}\xi-\scalar\alpha\xi)}\to a_{j,\alpha_j} \neq
0
\end{equation}
when ${k}\to +\infty$, and  
\begin{equation}
  \label{eq:8.5}
  f^*p(\zeta_{k})/(\zeta_{k}^{\alpha_1},\dots,\zeta_{k}^{\alpha_n})^\gamma
=(f_1(\zeta_{k})/\zeta_{k}^{\alpha_1})^{\gamma_1}\cdots
(f_n(\zeta_{k})/\zeta_{k}^{\alpha_n})^{\gamma_n}.
\end{equation}
By  combining (\ref{eq:8.3}), (\ref{eq:8.4}), and (\ref{eq:8.5}),
we see that 
$\big(f^*p(\zeta_{k})e^{-mH_T(\zeta_{k})}\big)_{{k}\in \N}$  is unbounded.
\end{proof}

\medskip
Assume now that $f$ is a proper map and that $q$ is a given admissible
weight function on a compact set $K\subseteq \C^n$.   Then $f^*q$ is lower
semicontinuous and
\begin{equation*}
\{z\in f^{-1}(K) \,;\, f^*q(z) <+\infty\}
=f^{-1}\big(\{w\in K\,;\, q(w)<+\infty \}\big).  
\end{equation*}
Since inverse images of non-pluripolar sets by proper maps are non-pluri\-polar
it follows that $f^*q$ is an admissible weight on $f^{-1}(K)$.
Furthermore, we have
\begin{equation*}
  \|f^*pe^{-mf^*q}\|_{f^{-1}(K)}=\|pe^{-mq}\|_K.
\end{equation*}
From Proposition \ref{prop:8.1} we conclude that
$f^*\big(\Phi^S_{K,q,m}\big) 
\leq \Phi^{S'}_{f^{-1}(K),f^*q,m}$, for every  $m\in \N^*$, 
consequently
$f^*\big(\Phi^S_{K,q}\big)\leq \Phi^{S'}_{f^{-1}(K),f^*q}$
and equality holds if 
$  f^*\colon {\mathcal P}^S_m(\C^n) \to  {\mathcal P}^{S'}_m(\C^\ell)$
is surjective. 

Next we look at the pullback of Lelong classes.
Let $u\in \L^S(\C^n)$, say $u\leq c_u+H_S$.
Then for every $z\in \C^\ell$ with $\Log f(z)\in \C^{*n}$ we have
\begin{align*}
  (f^*u)(z)&\leq c_u+H_S(f(z))
=c_u+\varphi_S(\log|f_1(z)|,\dots,\log|f_1(z)|)\\
&\leq c_u
+\varphi_S(c_1+\varphi_{S_1}(\Log z),\dots,c_n+\varphi_{S_n}(\Log z))
\nonumber
\\
&\leq c_u+\varphi_S(c)+\varphi_S(\varphi_{S_1}(\Log z),
\dots,\varphi_{S_n}(\Log z))
\nonumber
\\
&= c_u+\varphi_S(c)+H_{S'}(z)
\nonumber
\end{align*}
and conclude that $f^*u\in \L^{S'}(\C^\ell)$.
If $u\leq q$ on $K$, then $f^*u\leq f^*q$ on $f^{-1}(K)$, we have
$f^*\big(V^S_{K,q}\big)\leq V^{S'}_{f^{-1}(K),f^*q}$
and that equality holds if $f^*\colon \L^S(\C^n) \to \L^{S'}(\C^\ell)$
is surjective.

When $\ell=n$ we have the following weighted transformation rule, 
which generalizes Klimek \cite[Theorem 5.3.1]{Kli:1991}
and Perera \cite[Theorem 1]{Per:2023}:

\begin{proposition}\label{prop:8.2}
If  $f\colon \C^n\to \C^n$ is proper, the following are equivalent: 
\begin{enumerate}
\item[{\bf (i)}] The difference $f^*H_{S}  - H_{S'}$ is bounded, that is
  $f^*H_S\in \mathcal{L}^{S'}_+(\C^n)$. 
\item[{\bf (ii)}]   
For every compact set $K\subset \C^n$ and every
  admissible weight $q$ on $K$ we have 
$f^*V^{S}_{K, q}=V^{S'}_{f^{-1}(K),\, f^* q}$. 
     \end{enumerate}
\end{proposition}

\begin{proof}
\textbf{(i)$\Rightarrow$(ii):} 
Let $u\in \mathcal{L}^{S'}(\C^n)$ with $u|_{f^{-1}(K)}\leq f^*q$. 
By Klimek \cite[Theorem 2.9.26]{Kli:1991} the function
$v(z)= \max\{u(w)\colon w\in f^{-1}(z)\}$ is plurisubharmonic on
$\C^n$. Let $c$ be a constant such that 
$f^*H_S-H_{S'}\geq -c$. 
Then $v|_K\leq q$ and 
$$v(z) \leq \max_{w\in f^{-1}(z)}H_{S'}(w)+ c_u 
\leq \max_{w\in f^{-1}(z)}f^*H_{S}(w) + c+c_u = H_S(z)+c+c_u.$$ 
It follows that  $u\leq f^*v\leq f^*V^S_{K,q}$. 

\smallskip \noindent
 \textbf{(ii)$\Rightarrow$(i):} 
Let $c=\max_{w\in f^{-1}(\overline{\D}^n)} H_{S'}(w)$ 
and $c'=\max_{w\in f(\overline{\D}^n)} H_S(w)$. 
Then 
\begin{equation*}
H_{S'}-c\leq V^{S'}_{f^{-1}(\overline{\D}^n)}= f^*
V^{S}_{\overline{\D}^n}= f^*H_S
\end{equation*}
and
\begin{equation}
H_{S'}= V^{S'}_{\D^n}\geq V^{S'}_{f^{-1}(f(\overline{\D}^n))}  =
f^*V^S_{f(\overline{\D}^n)}\geq f^*H_S-c'.   
\end{equation}
\end{proof}

{\small 
\bibliographystyle{siam}
\bibliography{rs_bibref}

\noindent

}

\end{document}